\documentclass[oneside,10pt,reqno]{amsart}
\usepackage{amssymb,amsmath,amsthm,bbm,enumerate,multirow,hyperref,stmaryrd,mathrsfs,mathabx,amsaddr,upgreek,circledsteps}
\usepackage[shortlabels]{enumitem}
\usepackage{scalerel,stackengine}

\stackMath
\newcommand\reallywidehat[1]{%
\savestack{\tmpbox}{\stretchto{%
  \scaleto{%
    \scalerel*[\widthof{\ensuremath{#1}}]{\kern-.6pt\bigwedge\kern-.6pt}%
    {\rule[-\textheight/2]{1ex}{\textheight}}
  }{\textheight}%
}{0.5ex}}%
\stackon[1pt]{#1}{\tmpbox}%
}

\makeatletter
\@namedef{subjclassname@2020}{%
  \textup{2020} Mathematics Subject Classification}
\makeatother

\allowdisplaybreaks
\theoremstyle{definition}
\newtheorem{definition}{Definition}
\theoremstyle{theorem}
\newtheorem{proposition}[definition]{Proposition}
\newtheorem{lemma}[definition]{Lemma}
\newtheorem{theorem}[definition]{Theorem}
\newtheorem{corollary}[definition]{Corollary}
\numberwithin{equation}{section}
\numberwithin{definition}{section}
\theoremstyle{remark}
\newtheorem{remark}[definition]{Remark}
\newtheorem{question}[definition]{Question}
\newtheorem{example}[definition]{Example}

\newtheorem{extendedremark}[definition]{Extended remark}
\def\PP{\mathbb{P}}
\def\HH{\mathcal{H}}
\def\KK{\mathcal{K}}
\def\pr{\mathrm{pr}}

\def\BB{\mathcal{B}}
\def\BBB{\mathsf{B}}

\def\AA{\mathcal{A}}

\def\GG{\mathcal{G}}
\def\FF{\mathcal{F}}
\def\FFF{\mathsf{F}}
\def\dd{\mathrm{d}}

\def\LLL{\mathrm{L}}

\def\JJ{\mathcal{J}}

\def\at{\mathrm{at}}
\def\LL{\mathcal{L}}
\def\I{\mathrm{I}}
\def\Exp{\mathrm{Exp}}
\def\HHH{\mathsf{H}}
\def\TT{\mathcal{T}}
\def\sym{\Circled{s}}
\DeclareMathOperator*{\esssup}{ess\,sup}
\DeclareMathOperator*{\stronglim}{s-lim\,}

\begin{document}
\title{Fock structure of complete Boolean algebras of type I factors and of unital factorizations}

\author{Matija Vidmar}
\address{Department of Mathematics, Faculty of Mathematics and Physics, University of Ljubljana}
\email{matija.vidmar@fmf.uni-lj.si}

\begin{abstract}
The  factorizable vectors of a complete Boolean algebra of type I factors, acting on a separable Hilbert space, are shown to be total, resolving a conjecture of Araki and Woods. En route, the spectral theory of noise-type Boolean algebras of Tsirelson is cast in the noncommutative language of  ``factorizations with unit'' for which a muti-layered characterization of  being ``of Fock type'' is provided.
\end{abstract}

\thanks{Support from the  Slovenian Research and Innovation Agency (ARIS) under programmes Nos. P1-0402 \& P1-0448 is acknowledged. The author thanks an anonymous Referee for valuable suggestions in terms of exposition and especially for saving him from an error in (an earlier version of) the proof of Proposition~\ref{proposition:multiplicative-factorizable}.}

\keywords{Factorization; type I factor; completeness; Fock space; factorizable vector; commutative von Neumann algebra; spectrum; independence; spectral independence; stochastic noise; Boolean algebra.}

\subjclass[2020]{Primary: 46L53. Secondary: 60A99.} 

\date{\today}

\maketitle

\section{Introduction}

\subsection{Motivation, agenda, related literature}
Araki and Woods \cite{araki-woods} have studied   factorizations --- Boolean algebras of factors --- finding motivation chiefly in the work of Murray and von Neumann  \cite{murray-neumann} on direct factorizations of Hilbert spaces; and in lattices of von Neumann algebras arising in quantum theory \cite{araki,araki2,kadison-types-local,streater} (see also \cite{horuzhy} as a more recent reference). They established \cite[Theorem~4.1]{araki-woods} that a complete atomic Boolean algebra of type $\I$ factors is unitarily equivalent to the factorization that is attached canonically to a  ``discrete''   tensor product of Hilbert spaces, viz. to an incomplete/Guichardet tensor product. In the atomless case they demonstrated \cite[Theorem~6.1]{araki-woods} that  a complete  nonatomic Boolean algebra of type $\I$ factors, acting on a separable space, is unitarily equivalent to the factorization associated canonically to a ``continuous'' tensor product of Hilbert spaces, more precisely  to a symmetric Fock space over a direct integral w.r.t. an atomless measure, \emph{provided} a certain condition --- whose validity was (implicitly) conjectured but not proved --- holds, namely that the so-called factorizable vectors are total \cite[p.~210, just after Definition~6.1; p.~237, first paragraph]{araki-woods}. In both the discrete and the continuous case, or their combination, we  speak informally of the factorizations as being, briefly,  ``Fock''.

Tsirelson   \cite{tsirelson} investigates  noise-type Boolean algebras --- Boolean algebras of  sub-$\sigma$-fields of a  probability measure; called, for short, following \cite{vidmar-noise}, just noise Boolean algebras here  --- which is culmination of an extensive body of previous work on stochastic noises, see  \cite{vershik-tsirelson,tsirelson-nonclassical,picard2004lectures}  and the references therein. When the $\LLL^2$ space of the underlying probability is separable it is  shown \cite[Theorem~1.1]{tsirelson} that completeness of the noise Boolean algebra implies  that the  square-integrable so-called additive integrals generate the whole $\sigma$-field of the underlying probability (which is referred to as the property of classicality), settling an open problem going back to Feldman \cite[Problem~1.9]{feldman}. 

The Boolean algebras of sub-$\sigma$-fields of ``commutative'' (classical) probability are related naturally to a special case of the Boolean algebras of factors    from the ``noncommutative'' world of  operator algebras (the $\LLL^2$ space of the ambient probability becoming the underlying Hilbert space),  having also a ``unit''   (the constant one function  on the sample space). The two theories have enriched each other over the courses of their development. Factorizable vectors of the ``noncommutative'' correspond basically to square-integrable so-called multiplicative integrals of the ``commutative''. The square-integrable multiplicative integrals generate the same $\sigma$-field as do the additive ones \cite[Corollary~9.20]{vidmar-noise} --- the so-called stable $\sigma$-field --- them generating everything is thus in parallel to the totality of the factorizable vectors, and it is equivalent to the associated ``factorization with unit'' --- ``unital factorization'' for short --- being ``of Fock type'' \cite[Subsection~8.2]{vidmar-noise}. 

Any L\'evy process --- any process with stationary independent increments --- gives rise canonically to a classical atomless noise Boolean algebra and hence to a unital factorization of Fock type. A detailed analysis of the factorizations arising in this way is presented in \cite{Vershik2003FockFA}.  Two examples of non-classical  noise Boolean algebras, having however a non-trivial ``classical part'', come from  Warren's  noises of splitting \cite{warren}, and of stickiness (``made by a Poisson snake'') \cite{warren-sticky}. Their associated unital factorizations are not of Fock type. At the opposite side of the spectrum are the black noise Boolean algebras, whose only additive/multiplicative integrals are the constants (but the $\LLL^2$ space is not just the constants).  The most important examples of these come from the noises of coalescence \cite[Section~7]{picard2004lectures}, of the Brownian web \cite{ellis}, and of percolation \cite{ssg}. See also \cite{spectra-harris,watanabe,Jan08thenoise}. 

In this paper we, in  a manner of speaking, come full circle on the (non)commutative roundabout, by establishing that the condition for being  Fock  in the (atomless) noncommutative case --- the totality of the factorizable vectors --- follows indeed from completeness, at least in the case when the underlying Hilbert space is separable. The way this is achieved is through the study of unital factorizations which are apparently the closest noncommutative analogue of noise Boolean algebras. They are so close, in fact, that for them, as long as they are acting on a separable Hilbert space, we can delineate the conditions for their classicality --- for them being of Fock type --- in a manner that is completely on par with the commutative case (for the latter see Extended remark~\ref{extendedremark:noise-Boolean}), and which goes some way towards addressing the comment made by Tsirelson in \cite[p.~2]{tsirelson} that ``the noncommutative case is still waiting for a similar treatment''. Though, Tsirelson also succeedes in defining a noise-type completion \cite[Subsection~1.4]{tsirelson}, a natural extension of any noise Boolean algebra (of a separable probability) consisting of those $\sigma$-fields from its monotone closure that admit an independent complement; and he shows that an atomless noise Boolean \emph{sub}algebra has the same square-integrable additive integrals as does the ambient one \cite[Subsection~1.5]{tsirelson}. Analogues of these two significant findings for unital factorizations will not be pursued here.

Looking out further, this work is related to noncommutative extensions of stochastic noises \cite{HKK,kostler-speicher} on the one hand and to product systems of Hilbert spaces \cite{liebscher,arveson-fock-analogue}  on the other. In these quoted works however, the von Neumann algebras/Hilbert spaces are indexed by domains of the real line or by the positive real line, whereas here, in the spirit of \cite{araki-woods,tsirelson}, we want to treat the objects stripped to their bones, retaining only the structure that really matters for the problem at hand  (which the indexation does not). We will comment more precisely on the connections  of unital factorizations to \cite{HKK,liebscher} in Remark~\ref{remark:connection} below, but in any event the issue of characterizing completeness through the factorization being  Fock was not addressed therein.

\subsection{Outline of main results}
 Throughout this subsection  $\HH$ is a separable non-zero Hilbert space over the complex number field. 
 
 In order to be able to report on the highlights of this paper in greater detail, let us be a little bit more precise about some of the notions which we have referred to in the above, deferring full rigor and generality to later. For basic notions of von Neumann algebras we refer the reader  to, e.g., \cite{bing-ren,dixmier1981neumann}. 

  Interpreting the commutant operation as complementation, the collection of all multiples of the identity on $\HH$ as $0$, that of all bounded linear operators on $\HH$ as $1$, finally intersection as meet $\land$ and generated von Neumann algebra as join $\lor$, a factorization $\FFF$ of $\HH$ is a family of factors of $\HH$, which is a Boolean algebra under the indicated constants and operations. Such $\FFF$ is said to be of type $\I$ when all its elements are type $\I$ factors.  A factorizable vector of a factorization $\FFF$ is a non-zero element $\xi$ of $\HH$ such that the projection onto $\xi$ can be decomposed into a product of the form $XX'$ with $X\in x$, $X'\in x'$ for all $x\in \FFF$. 
  
The continuous Fock factorization associated to a direct integral $\int^\oplus \GG_s\nu(\dd s)$  of Hilbert spaces, $\nu$ standard atomless, is the factorization of the symmetric Fock space $\Gamma_s(\int^\oplus \GG_s\nu(\dd s))=\oplus_{n\in \mathbb{N}_0}\left(\int^\oplus \GG_s\nu(\dd s)\right)^{\sym n}$, comprising precisely the  factors which consist of the ``copies'' of the bounded linear operators on  $\Gamma_s(\int^\oplus_A \GG_s\nu(\dd s))$ in $\Gamma_s(\int^\oplus \GG_s\nu(\dd s))$,  $A$ running over the $\nu$-measurable sets; here copies are understood via the factorization $\Gamma_s(\int^\oplus \GG_s\nu(\dd s))=\Gamma_s(\int^\oplus_A \GG_s\nu(\dd s))\otimes \Gamma_s(\int^\oplus_{A^{\mathrm{c}}} \GG_s\nu(\dd s))$  valid up to the natural unitary equivalence coming out of the orthogonal decomposition $\int^\oplus \GG_s\nu(\dd s)=\int^\oplus_A \GG_s\nu(\dd s)\oplus \int^\oplus_{A^{\mathrm{c}}}  \GG_s\nu(\dd s)$, the copies acting as the identity on the second factor. The number $1\in \mathbb{C}=\left(\int^\oplus \GG_s\nu(\dd s)\right)^{\otimes 0}\subset \Gamma_s(\int^\oplus \GG_s\nu(\dd s))$ is referred to as the vacuum vector. A slightly different construction is more natural for a purely atomic standard $\nu$ giving a discrete Fock factorization. In general we take combination (product) of the factorizations coming out of the discrete and continuous part of a direct integral $\int^\oplus \GG_s\nu(\dd s)$, $\nu$ standard, referring to it simply as a Fock factorization. 

With the preceding notions having been introduced in rough form, a somewhat hazed panorama of the findings of this paper is as follows. 

We shall prove in Theorem~\ref{thm:main-two} and Corollary~\ref{cor:main} that 
\begin{quote}
for a complete type $\I$ factorization $\FFF$ of $\HH$ its factorizable vectors are total and such $\FFF$ is therefore unitarily isomorphic to a Fock factorization, which is continuous or discrete respectively as $\FFF$ is atomless or atomic.
\end{quote}

Two essential ingredients will be used in reaching the preceding conclusion. The first is intermediate, namely that a complete type $\I$ factorization $\FFF$ of $\HH$ always admits at least one norm one factorizable vector $\Omega$. The latter enhances $\FFF$ into the unital factorization $(\FFF,\Omega)$, $\Omega$ being the unit. 

Being now in the realm of $(\FFF,\Omega)$, the second key ingredient  in the proof of the result in display above is the implication  \ref{intro-a} $\Rightarrow$ \ref{intro-d} of what is otherwise the equivalence of the following statements, which is itself in partial summary of Theorem~\ref{thm:main-for-noise-factorization} that is a substantial finding in its own right.
\begin{enumerate}[(a)]
\item\label{intro-a} There is a complete factorization $\overline{\FFF}$, such that $\FFF\subset \overline{\FFF}$ and $(\overline{\FFF},\Omega)$ is a unital factorization.
\item\label{intro-b}   For every sequence $(x_n)_{n\in \mathbb{N}}$ in $\FFF$, $\lor_{n\in \mathbb{N}}x_n$ is a factor.
\item\label{intro-c} $(\FFF,\Omega)$ is unitarily equivalent to a unital subfactorization of a Fock factorization equipped with the vacuum vector as unit.
\item\label{intro-d} The factorizable vectors $\xi$ of $\FFF$ satisfying $\langle\Omega,\xi\rangle=1$ are total.
\end{enumerate}

To avoid too heavy repetition, we leave it in the above to the reader to guess at the (almost obvious) meanings of  ``complete'', ``unitarily equivalent'' and of  ``unital subfactorization''.

\subsection{Article structure}Section~\ref{section:setting}  delineates precisely factorizations and their factorizable vectors. Then Section~\ref{section:unital-factorizations} introduces unital factorizations, describes the more   rudimentary elements of their landscape and gives the important Fock examples. We continue in Section~\ref{section:spectrum} by considering the spectrum of a unital factorization, which is a main tool in proving the results of Section~\ref{section:fock-structure-completneess}. The latter include, but are not limited to the equivalences between \ref{intro-a}-\ref{intro-d} above, or even to those of Theorem~\ref{thm:main-for-noise-factorization}. We draw attention especially to Theorem~\ref{thm:black} concerning the analogue of ``blackness'' from the commutative world. Last but not least, Section~\ref{section:factorizations-sans} deals with factorizations per se, no a priori unit,  culminating in establishing the announced totality of the factorizable vectors in the complete type $\I$ case. 

\subsection{Miscellaneous general notation, conventions} All Hilbert spaces appearing herein are complex. Inner products are linear in the second argument, conjugate linear in the first.  For a Hilbert space $\mathcal{K}$: $e_\mathcal{K}(z):=\oplus_{n\in \mathbb{N}_0}\frac{1}{\sqrt{n!}}z^{\otimes n}$ denotes the exponential  vector of the symmetric Fock space $\Gamma_s(\mathcal{K}):=\oplus_{n\in \mathbb{N}_0}\mathcal{K}^{\sym n}$ ($\mathcal{K}^{\sym n}$ being the $n$-fold symmetric tensor product of $\KK$ with itself) associated to a $z\in \mathcal{K}$; $\alpha\KK$ is $\KK$ endowed with the scalar product $\alpha\langle \cdot,\cdot\rangle$, where $\langle \cdot,\cdot\rangle$ is the given scalar product on $\KK$, $\alpha\in (0,\infty)$; $\mathbf{1}_\KK$ is the identity on $\KK$; for $\xi\in \KK$, $\vert \xi\rangle\langle \xi\vert:=(\KK\ni z\mapsto \langle \xi,z\rangle \xi)$, which is the projection onto the one-dimensional space $\mathbb{C}\xi$ spanned by $\xi$, when $ \xi$ has norm one; $\BB(\KK)$ is the collection of all bounded linear operators on $\KK$;  $[\GG]$ denotes the orthogonal projection onto a closed linear subspace $\GG$ of $\KK$; $x':=\{X'\in \BB(\KK):X'X=XX'\text{ for all }X\in x\}$ is the commutant of an $x\subset \BB(\KK)$.

Given a measure $\mu$ on a measurable space $(S,\Sigma)$, $\mu[f]:=\int f\dd\mu$ (resp. $f\cdot \mu:=(\Sigma\ni A\mapsto \mu[f;A]:=\mu[f\mathbbm{1}_A])$) denotes the (resp. indefinite) integral of a $\Sigma$-measurable numerical $f$ against $\mu$, whenever it has significance (is well-defined); assuming $\mu$ has measurable singletons and is $\sigma$-finite, $\at(\mu):=\{s\in S:\mu(\{s\})>0\}$  and $\mu_c:=\mu\vert_{S\backslash \at(\mu)}:=\mu\vert_{\Sigma\vert_{S\backslash \at(\mu)}}$ are its collection of atoms and continuous (diffuse) part, respectively; $g_\star\mu:=(\Delta\ni A\mapsto \mu(g^{-1}(A)))$ is the push-forward of $\mu$ along a $\Sigma/\Delta$-measurable $g$. By a standard $\sigma$-finite measure we mean the completion of a $\sigma$-finite measure on a standard measurable space, possibly adjoined by an arbitrary null set (and completed again).

A partition of unity of a Boolean algebra $B$ is a finite subset $P$ of $B\backslash \{0\}$ such that $\lor P=1$ and $x\land y=0$ for $x\ne y$ from $P$ ($P$ may be empty, which happens only if $0=1$, a case that we allow).  The collection of atoms of a Boolean algebra $B$ is denoted $\at(B)$. 

A choice function on a collection of sets $P$ is a map $f$ with domain $P$ satisfying $f(p)\in p$ for all $p\in P$. For a set $A$, $2^A$ is its power set, $(2^A)_{\mathrm{fin}}:=\cup_{n\in\mathbb{N}_0}{A\choose n}$ are its finite subsets and, in case $A$ is finite, $\vert A\vert$ is its size ($\#$ of elements). $[n]:=\{1,\ldots,n\}$ ($=\emptyset$ when $n=0$) and $\mathbb{N}_{>n}:=\{n+1,n+2,\ldots\}$ for $n\in \mathbb{N}_0$,  while $\mathbb{N}_{\geq m}:=\{m,m+1,\ldots\}$ for $m\in \mathbb{N}$. 

The value of a function $f$ at a point $x$ of its domain is written as $f(x)$ or $f_x$, whichever is the more convenient. The symbol $\uparrow$ (resp. $\uparrow\uparrow$) means nondecreasing (resp. strictly increasing); analogously we interpret $\downarrow$, $\downarrow\downarrow$. We set $\inf\emptyset:=\infty$, $\max\emptyset:=0$, $\sum_\emptyset:=0$ (empty sum) and $\prod_\emptyset:=1$ (empty product), as usual.

\section{Factorizations of a Hilbert space and their factorizable vectors}\label{section:setting}

We fix henceforth  a non-zero Hilbert space $\HH$ and let 
\begin{equation*}
\hat{\HH}:=\{x\in 2^{\BB(\HH)}:x\text{ a von Neumann algebra on }\HH\}
\end{equation*} 
denote the bounded, involutive lattice of all von Neumann algebras  on $\HH$: the lattice operations
are the meet  $\land=$ intersection and join $\lor$ = the von Neumann algebra generated by the
union, involution ${}'$ is  passage to the commutant $\hat\HH\ni x\mapsto x'$, the bottom
element is the von Neumann algebra $0_\HH:=\mathbb{C}\mathbf{1}_\HH$ of the scalar multiples of the identity
operator $\mathbf{1}_\HH$, and the top element is the von Neumann algebra $1_\HH:=\BB(\HH)$ of all
bounded linear operators on $\HH$.   

\begin{remark}
The involutive properties of $\hat\HH$, namely that $x''=x$ and $(x\lor y)'=x'\land y'$ for $\{x,y\}\subset \hat\HH$, follow from the double commutant theorem; actually, $\land_{\alpha\in \mathfrak{A}}x_\alpha'=(\cup_{\alpha\in \mathfrak{A}}x_\alpha)'=(\cup_{\alpha\in \mathfrak{A}}x_\alpha)'''=((\cup_{\alpha\in \mathfrak{A}}x_\alpha)'')'=(\lor_{\alpha\in \mathfrak{A}} x_\alpha)'$ for any family $(x_\alpha)_{\alpha\in \mathfrak{A}}$ in $\hat \HH$. The commutation theorem for tensor products of von Neumann algebras \cite[Theorem~1]{commutant-tensor} entails that for non-zero Hilbert spaces $\KK$ and $\FF$ the map $\hat \KK\times \hat\FF\ni (k,f)\mapsto k\otimes f\in  \widehat{\KK\otimes \FF}$ is a homomorphism of   bounded involutive lattices. Besides,  $\{k\otimes 0_\FF:k\in\hat\KK\}=\widehat{\KK\otimes \FF}\cap 2^{1_\KK\otimes 0_\FF}$ and, for $k\in \hat{\KK}$, $k\otimes 0_\FF=\{K\otimes \mathbf{1}_\FF:K\in k\}$, the ``copy'' of $k$ in $\KK\otimes \FF$. 
 We will use this without special mention in what follows. 
\end{remark}

\subsection{Factorizations}
We remind the reader  that a factor of $\HH$ is a von Neumann algebra $x$ on $\HH$ such that $x\lor x'=1_\HH$,  equivalently $x\land x'=0_\HH$; and that it  is said to be of type $\I$ when it contains a non-zero minimal projection, which is equivalent to it being (algebraically) isomorphic to $\BB(\KK)$ for some non-zero Hilbert space $\KK$. The commutant of a factor is a factor, of course, but the meet and join of two factors in general need not be factors.  If a non-empty sublattice $\mathsf{O}\subset \hat\HH$ consisting of factors is closed for the commutant operation (is involutive), then $\mathsf{O}$ has the structure of an orthocomplemented lattice.  Insisting further on distributivity we arrive at the notion of
\begin{definition}\label{definition:factorization}
A factorization of  $\HH$ is a distributive involutive non-empty sublattice $\FFF$ of $\hat{\HH}$ consisting of factors, distributivity meaning that 
\begin{equation*}
x\land (y\lor z)=(x\land y)\lor (x\land z),\quad \{x,y,z\}\subset\FFF,
\end{equation*} or equivalently (by involution) 
\begin{equation*}
x\lor (y\land z)=(x\lor y)\land (x\lor z),\quad \{x,y,z\}\subset\FFF.
\end{equation*} In such case:
\begin{itemize}
\item $\FFF$ is of type $\mathrm{I}$ if all its elements are type $\mathrm{I}$ factors; 
\item $\FFF$ is complete if it is closed under arbitrary joins of its subfamilies;
\item the atoms of $\FFF$ are the atoms of the Boolean algebra that it is (see Remark~\ref{rmk:boolean-algebra} just below), similarly we interpret partitions of unity, atomless/nonatomic and atomic;
\item a subfactorization of $\FFF$ is a subset of $\FFF$ that is itself a factorization (of the same Hilbert space). 
\end{itemize}
Given  factorizations $\FFF_1$ and $\FFF_2$ of the  Hilbert spaces $\HH_1$ and $\HH_2$, respectively, we: 
\begin{itemize}
\item define their product  as the factorization $\FFF_1\otimes \FFF_2:=\{x_1\otimes x_2:(x_1,x_2)\in \FFF_1\times\FFF_2\}$ of $\HH_1\otimes \HH_2$; 
\item say $\FFF_1$ and $\FFF_2$ are isomorphic if there is a unitary isomorphism between $\HH_1$ and $\HH_2$ that carries $\FFF_1$ onto $\FFF_2$. %
\end{itemize}
\end{definition}
\begin{remark}\label{rmk:boolean-algebra}
 We have (implicitly) insisted that a factorization always acts on a non-zero Hilbert space; the degenerate case corresponding to a zero-dimensional space, albeit technically possible, is excluded. Interpreting $1_\HH$ as $1$, $0_\HH$ as $0$ and the commutant operation as complementation, together with the meet and join as above, a factorization of $\HH$ is a Boolean algebra. Put more succinctly then, a factorization is a Boolean algebra of factors.  If a factorization is complete, then it is closed also for arbitrary meets of its subfamilies (we see it by involution).
\end{remark}

\subsection{A certain kind of independence}\label{subsction:kind-of-independence} Throughout this subsection $\Omega\in \HH$ is a norm one vector.
In the following it would be more informative to say ``independent under the vector state $\langle\Omega,\cdot\Omega\rangle$'' but we yield to brevity of expression and say just ``independent under $\Omega$''. 

\begin{definition}\label{def:independence}
Let $\{x,y\}\subset \hat\HH$. We say $x$ and $y$ are independent under $\Omega$ when $x$ and $ y$ are commuting on $\overline{(x\lor y)\Omega}$ (= the closure of $\{Z\Omega:Z\in x\lor y\}$ in the Hilbert space topology of $\HH$) and $\langle \Omega,XY\Omega\rangle=\langle \Omega,X\Omega\rangle\langle \Omega,Y\Omega\rangle$ for all $X\in x$ and $Y\in y$. 
\end{definition}

%
%
%

\begin{example}\label{example:classical}
For a probability $\PP$ consider the Hilbert space $\LLL^2(\PP)$. Identifying the members of $\LLL^\infty(\PP)$ with their multiplication maps, $\LLL^\infty(\PP)$ is a commutative von Neumann algebra on $\LLL^2(\PP)$. Its von Neumann subalgebras are of the form $\LLL^\infty(\PP\vert_x)$ for a complete sub-$\sigma$-field $x$ of $\PP$. For two complete sub-$\sigma$-fields $x$ and $y$ of $\PP$, $\LLL^\infty(\PP\vert_x)$ and $\LLL^\infty(\PP\vert_y)$ are independent under $\mathbbm{1}$ ($:=$ the constant one function on the sample space of $\PP$) iff $x$ and $y$ are independent under $\PP$ in the usual sense. 
\end{example}
Recall that we use  $[\GG]$ to denote the orthogonal projection onto a closed linear subspace $\GG$ and that $\vert \Omega\rangle\langle\Omega\vert=[\mathbb{C}\Omega]$ is the projection onto the one-dimensional subspace spanned by $\Omega$.
\begin{proposition}\label{quantum-independence}
Let $\{x,y\}\subset \hat\HH$. The following are equivalent.
\begin{enumerate}[(i)]
\item\label{quantum-independence:A} $x$ and $y$ are independent under $\Omega$.
\item\label{quantum-independence:B} $x$ and $y$ are commuting on $\overline{(x\lor y)\Omega}$ and $\vert \Omega\rangle\langle\Omega\vert=[\overline{x\Omega}][\overline{y\Omega}]$ (and then automatically also $=[\overline{y\Omega}][\overline{x\Omega}]$).
\item\label{quantum-independence:C} There exists a (then automatically unique)  unitary isomorphism between $\overline{(x\lor y)\Omega}$ and $\overline{x\Omega}\otimes \overline{y\Omega}$ that sends $\Omega$ to $\Omega\otimes \Omega$ and  $(XY)\vert_{\overline{(x\lor y)\Omega}}$ to  $X\vert_{\overline{x\Omega}}\otimes Y\vert_{\overline{y\Omega}}$ for all $X\in x$ and all $Y\in y$ (and then automatically it sends: $\overline{x\Omega}$ onto $\overline{x\Omega}\otimes (\mathbb{C}\Omega)$ and $\overline{y\Omega}$ onto $(\mathbb{C}\Omega)\otimes \overline{y\Omega}$; also $x\vert_{\overline{(x\lor y)\Omega}}$ into $1_{\overline{x\Omega}}\otimes 0_ {\overline{y\Omega}}$ and $y\vert_{\overline{(x\lor y)\Omega}}$ into  $0_{\overline{x\Omega}}\otimes 1_{\overline{y\Omega}}$).
\end{enumerate}
\end{proposition}
\begin{proof}
The paranthetical additions of \ref{quantum-independence:B}-\ref{quantum-independence:C} are  clear as is the implication  \ref{quantum-independence:C}$\Rightarrow$\ref{quantum-independence:B}. Assuming \ref{quantum-independence:A}, the relation which associates
\begin{equation*} XY\Omega\text{ with }(X\Omega)\otimes (Y\Omega) \text{ for }(X,Y)\in x\times y
\end{equation*}
 is scalar product-preserving between total sets in $\overline{(x\lor y)\Omega}$ and $\overline{x\Omega}\otimes \overline{y\Omega}$, hence  extends uniquely to a unitary isomorphism, which gives \ref{quantum-independence:C}. If  \ref{quantum-independence:B} holds true, then for $X\in x$ and $Y\in y$, 
 \begin{equation*}\langle \Omega,XY\Omega\rangle=\langle X^*\Omega,Y\Omega\rangle=\langle [\overline{x\Omega}]X^*\Omega,[\overline{y\Omega}]Y\Omega\rangle=\langle X^*\Omega,[\overline{x\Omega}][\overline{y\Omega}]Y\Omega\rangle=\langle X^*\Omega,\Omega\rangle\langle \Omega,Y\Omega\rangle=\langle\Omega,X\Omega\rangle\langle \Omega,Y\Omega\rangle;
 \end{equation*}
  together with the assumed commutativity of $x$ and $y$ on $\overline{(x\lor y)\Omega}$ we get \ref{quantum-independence:A}. 
\end{proof}
%
The preceding may be compared with \cite[Subsection~2.2]{HKK} which features a very similar notion of independence. It is extremely well-behaved. For instance, in addition to the characterizations of Proposition~\ref{quantum-independence} we have also

\begin{proposition}\label{proposition:raise-independence} Let $x$ and $y$ be self-adjoint multiplicative subsets of $1_\HH$. Suppose $XY\Omega=YX\Omega$ and  $\langle \Omega,XY\Omega\rangle=\langle \Omega ,X\Omega\rangle\langle \Omega,Y\Omega\rangle$ for all $X\in x$ and $Y\in y$. Then $x''$ and $y''$  are independent under $\Omega$.
 \end{proposition}
In passing it may be compared with, indeed it subsumes the ``classical'' probability analogue: independence can be raised from $\pi$-systems to the $\sigma$-fields they generate (which follows by itself more easily from Dynkin's lemma).
\begin{proof}
We may and do assume $\mathbf{1}_\HH\in x\cap y$. Then, since $x$ and $y$ are multiplicative and self-adjoint, applying the double commutant theorem, we infer that $x''$, $y''$ are the strong operator closures of the linear spans of $x$, $y$, respectively. In virtue of the same token, $x''\lor y''$ is the strong operator closure of the linear span of $$\mathcal{Z}:=\cup_{n\in \mathbb{N}_0}\{X_nY_n\cdots X_1Y_1:(X_1,\ldots,X_n)\in x^n,\, (Y_1,\ldots,Y_n)\in y^n\}.$$

By linearity and continuity, since $\langle \Omega,XY\Omega\rangle=\langle \Omega ,X\Omega\rangle\langle \Omega,Y\Omega\rangle$ holds  for all $X\in x$ and $Y\in y$ by assumption, it holds: first for all $X\in x''$ and $Y\in y$; second, for all $X\in x''$ and $Y\in y''$.
 
As for the  ``commutativity property'', we verify that
\begin{equation}\label{eq:raise-independence-proof}
XY(X_nY_n\cdots X_1Y_1\Omega)=YX(X_nY_n\cdots X_1Y_1\Omega)\text{ whenever } \{X,X_1,\ldots,X_n\}\subset x,\, \{Y,Y_1,\ldots,Y_n\}\subset y,
\end{equation}
holds for all $n\in \mathbb{N}_0$, which we see inductively: \eqref{eq:raise-independence-proof} is true for $n=0$ (to wit, $XY\Omega=YX\Omega$ for all $X\in x$, all $Y\in y$) by assumption; if \eqref{eq:raise-independence-proof} holds for a given $n\in \mathbb{N}_0$, then for arbitrary $ \{X,X_1,\ldots,X_{n+1}\}\subset x$ and $\{Y,Y_1,\ldots,Y_{n+1}\}\subset y$ we deduce
 \begin{align*}
 XY(X_{n+1}Y_{n+1}\cdots X_1Y_1\Omega)&=X(YY_{n+1})X_{n+1}(X_nY_n\cdots X_1Y_1\Omega)\\
 &=YY_{n+1}(XX_{n+1})(X_nY_n\cdots X_1Y_1\Omega)\\
  &=YX(X_{n+1}Y_{n+1}X_nY_n\cdots X_1Y_1\Omega),
 \end{align*}
 so \eqref{eq:raise-independence-proof} holds true for $n+1$ also.  \eqref{eq:raise-independence-proof} being true for all $n\in \mathbb{N}_0$ means that any $X\in x$ and any $Y\in y$ are commuting on $\mathcal{Z}\Omega$; therefore, by continuity and linearity they are commuting on the closure of the linear span of $\mathcal{Z}\Omega$, i.e. on $\overline{(x''\lor y'')\Omega}$. Yet again by linearity and continuity we find in turn that $x''$ and $y$ are commuting on $\overline{(x''\lor y'')\Omega}$, and then finally that $x''$ and $y''$ are commuting on $\overline{(x''\lor y'')\Omega}$.
\end{proof}

\begin{example}
Returning to Example~\ref{example:classical}, for independent complete sub-$\sigma$-fields $x$ and $y$ of $\PP$, the unitary isomorphism of Proposition~\ref{quantum-independence}\ref{quantum-independence:C} associated to $\LLL^\infty(\PP\vert_x)$ and $\LLL^\infty(\PP\vert_y)$ on $ \HH=\LLL^2(\PP)$ and with $\Omega=\mathbbm{1}$ is just the natural unitary equivalence between $\LLL^2(\PP\vert_{x\lor y})$ and $\LLL^2(\PP\vert_x)\otimes\LLL^2(\PP\vert_y)$ sending $fg$ to $f\otimes g$ for $(f,g)\in\LLL^2(\PP\vert_x)\times\LLL^2(\PP\vert_y)$.
\end{example}

\subsection{Factorizable vectors}\label{subsection:factorizable}
We now specialize Proposition~\ref{quantum-independence} to $y=x'$, $x$ a factor of $\HH$, and elaborate on it  further in that case. Ultimately, a non-zero vector $\omega\in \HH$ shall be called factorizable relative to a factorization $\FFF$ when  its normalization $\Omega:=\frac{\omega}{\Vert\omega\Vert}$ satisfies one (and then all) of the conditions of the ensuing Proposition~\ref{proposition:factorizable-myriad} \emph{for all} factors $x$ belonging to $ \FFF$.  

A non-trivial result shall be used in the proof of Proposition~\ref{proposition:factorizable-myriad}, namely: \begin{quote}
$(\otimes)$ For a type $\I$ factor $x$ of $\HH$ there is, up to unitary isomorphism, a tensor product decomposition $\HH=\GG\otimes \GG'$ into non-zero Hilbert spaces $\GG$ and $\GG'$ so that $x=1_{\GG}\otimes 0_{\GG'}$ and (hence) $x'=0_\GG\otimes 1_{\GG'}$. \label{otimes-page}
\end{quote}
We refer to  \cite[combine: Definitions~2.3.1,~ 3.1.1 and~3.2.1; Theorem~IV on p.~150; Lemmas~5.3.1,~5.4.1]{murray-neumann}, or  \cite[Proposition~I.2.5, taking for $A$ a type $\I$ factor and for $(E_i)_{i\in I}$ the minimal non-zero projections of $A$]{dixmier1981neumann}, or \cite[Proposition~I.5.6, taking for $M$ a type $\I$ factor and for $(p_l)_{l\in \Lambda}$ the minimal non-zero projections of $M$]{bing-ren} for $(\otimes)$.

\begin{proposition}\label{proposition:factorizable-myriad}
Let $\Omega$ be a norm one vector and $x$ a factor of $\HH$. The following are equivalent.
\begin{enumerate}[(i)]
\item\label{proposition:factorizable-myriad:i} $x$ and $x'$ are independent under $\Omega$, i.e.   $\langle\Omega,XX'\Omega\rangle=\langle\Omega,X\Omega\rangle\langle\Omega,X'\Omega\rangle$ for all $X\in x$, $X'\in x'$.
\item\label{proposition:factorizable-myriad:ii} $\vert\Omega \rangle\langle\Omega\vert=[\overline{x'\Omega}][\overline{x\Omega}]$ (and then necessarily also $=[\overline{x\Omega}][\overline{x'\Omega}]$). 
\item\label{proposition:factorizable-myriad:iii} There exists a  unitary isomorphism between $\HH$ and $\overline{x\Omega}\otimes \overline{x'\Omega}$ that sends $\Omega$ to $\Omega\otimes \Omega$ and  $XX'$ to  $X\vert_{\overline{x\Omega}}\otimes X'\vert_{\overline{x'\Omega}}$ for all $X\in x$ and all $X'\in x'$ (and then necessarily: the unitary isomorphism in question is unique; $\overline{x\Omega}$ is sent by it onto $\overline{x\Omega}\otimes (\mathbb{C}\Omega)$ and $\overline{x'\Omega}$ onto $(\mathbb{C}\Omega)\otimes \overline{x'\Omega}$; also $x$ is carried by it onto  $1_{\overline{x\Omega}}\otimes 0_ {\overline{x'\Omega}}$ and $x'$ onto  $0_{\overline{x\Omega}}\otimes 1_{\overline{x'\Omega}}$).
\item\label{proposition:factorizable-myriad:iv} There exist Hilbert spaces $\GG$ and $\GG'$, norm one vectors $\Theta\in \GG$ and $\Theta'\in \GG'$ and a unitary isomorphism between $\HH$ and $\GG\otimes \GG'$ that maps $\Omega$ to $\Theta\otimes \Theta'$ and carries $x$ onto $1_{\GG}\otimes 0_{\GG'}$ and $x'$ onto $0_{\GG}\otimes 1_{\GG'}$. 
\item\label{proposition:factorizable-myriad:vi} There exist minimal projections $P$ and $P'$ of $x$ and $x'$ respectively, such that $\vert \Omega \rangle\langle\Omega\vert=PP'$ (and then necessarily: $P=[\overline{x'\Omega}]$ and $P'=[\overline{x\Omega}]$).
\item\label{proposition:factorizable-myriad:vii} There exist  minimal projections $P$ of $x$ and $P'$ of $x'$ satisfying $P\Omega=\Omega$ and $P'\Omega=\Omega$ (and then necessarily $P=[\overline{x'\Omega}]$ and $P'=[\overline{x\Omega}]$). 
\item\label{proposition:factorizable-myriad:v} 
 For some $X\in x$ and some $X'\in x'$ we have $\vert \Omega \rangle\langle\Omega\vert=XX'$. 
\end{enumerate}
\end{proposition} 
\begin{proof}
\ref{proposition:factorizable-myriad:i}$\Leftrightarrow$\ref{proposition:factorizable-myriad:ii}$\Leftrightarrow$\ref{proposition:factorizable-myriad:iii}. On noting that $x\lor x'=1_\HH$ and hence $\overline{(x\lor x')\Omega}=\HH$ and that of course $x$ commutes with $x'$,  this is essentially a rewriting of the same equivalences of  Proposition~\ref{quantum-independence} applied with $y=x'$ and $x$, $\Omega$ as themselves, except that for the last part of the paranthetical addition of \ref{proposition:factorizable-myriad:iii} we notice as follows: on the one hand, the von Neumann algebra $x$ is carried into $1_{\overline{x\Omega}}\otimes 0_ {\overline{x'\Omega}}$; on the other hand, $x'$ is carried into $0_{\overline{x\Omega}}\otimes 1_{\overline{x'\Omega}}$; hence, the image of $x=x''$ is not only contained in, but also contains  $(0_{\overline{x\Omega}}\otimes 1_{\overline{x'\Omega}})'=1_{\overline{x\Omega}}\otimes 0_ {\overline{x'\Omega}}$.

\ref{proposition:factorizable-myriad:iii}$\Rightarrow$\ref{proposition:factorizable-myriad:iv}. We have only to take $\GG=\overline{x\Omega}$, $\GG'=\overline{x'\Omega}$, $\Theta=\Theta'=\Omega$ and for the unitary isomorphism the one of \ref{proposition:factorizable-myriad:iii}. The implication \ref{proposition:factorizable-myriad:iv}$\Rightarrow$\ref{proposition:factorizable-myriad:i} is trivial.

At this point we are assured the equivalence of \ref{proposition:factorizable-myriad:i}-\ref{proposition:factorizable-myriad:iv}.  We shall dispose of the paranthetical additions of \ref{proposition:factorizable-myriad:vi}-\ref{proposition:factorizable-myriad:vii} once we have established the equivalences \ref{proposition:factorizable-myriad:i}-\ref{proposition:factorizable-myriad:vii}.


(\ref{proposition:factorizable-myriad:ii}\&(\ref{proposition:factorizable-myriad:iii})$\Rightarrow$ \ref{proposition:factorizable-myriad:vi}. We take $P=[\overline{x'\Omega}]$ and $P'=[\overline{x\Omega}]$. 

 \ref{proposition:factorizable-myriad:vi}$\Rightarrow$\ref{proposition:factorizable-myriad:vii}.  $PP'\Omega=\vert\Omega \rangle\langle\Omega\vert\Omega=\Omega$, so $P\Omega=PPP'\Omega=PP'\Omega=\Omega$.

\ref{proposition:factorizable-myriad:vii}$\Rightarrow$\ref{proposition:factorizable-myriad:i}. Since the factors $x$ and $x'$ contain the non-zero minimal projections $P$ and $P'$ respectively, they are both type $\I$. Consequently, by $(\otimes)$, for some non-zero Hilbert spaces $\GG$ and $\GG'$, up to unitary isomorphism: $\KK=\GG\otimes \GG'$, $x=1_{\GG}\otimes 0_{\GG'}$ and  therefore $PP'=Q\otimes Q'$ for some one-dimensional projections $Q$ and $Q'$ of $\GG$ and $\GG'$ respectively. But $P\Omega=\Omega$ and $P'\Omega=\Omega$, so $PP'\Omega=\Omega$, which entails $[\mathbb{C}\Omega]\leq PP'$. Hence, still up to said unitary isomorphism, the one-dimensional projection $PP'=Q\otimes Q'$ is onto $[\mathbb{C}\Omega]$, i.e. $\Omega=\Theta\otimes \Theta'$ for some norm one vectors $\Theta\in\GG$ and $\Theta'\in \GG'$. Now  \ref{proposition:factorizable-myriad:i} follows at once. 

Thus \ref{proposition:factorizable-myriad:i}-\ref{proposition:factorizable-myriad:vii} are equivalent.

Uniquenesses of  $P$, $P'$ in \ref{proposition:factorizable-myriad:vi}. $\vert\Omega \rangle\langle\Omega\vert=PP'$ necessitates $x'\Omega=x'PP'\Omega=Px'P'\Omega\subset P\HH$ so $[\overline{x'\Omega}]\leq P$; since $P$ is minimal in $x$ and since the projection $[\overline{x'\Omega}]$, which belongs to $x$ by \ref{proposition:factorizable-myriad:iii}, is not zero, we conclude that $P=[\overline{x'\Omega}]$; analogously we deduce that $P'=[\overline{x\Omega}]$. 

Uniquenesses of $P$, $P'$ in \ref{proposition:factorizable-myriad:vii}.  $P\Omega=\Omega$ entails $x'\Omega=x'P\Omega=Px'\Omega\subset P\HH$, and we continue and finish exactly as for \ref{proposition:factorizable-myriad:vi}.

The implication \ref{proposition:factorizable-myriad:vi}$\Rightarrow$\ref{proposition:factorizable-myriad:v} is trivial.

\ref{proposition:factorizable-myriad:v}$\Rightarrow$\ref{proposition:factorizable-myriad:ii}. $\vert\Omega \rangle\langle\Omega\vert=XX'$ implies $\vert\Omega \rangle\langle\Omega\vert=(XX')^*(XX')=X^*X{X'}^*X'$. Since $\vert\Omega \rangle\langle\Omega\vert$ is a one-dimensional projection, the spectral measures of the non-zero commuting Hermitean operators $X^*X\in x$ and  ${X'}^*X'\in x'$ must have for their supports, stripped of $0$, singletons $\{\lambda\}$ and $\{\lambda'\}$ respectively with $\lambda\lambda'=1$. It follows that $\vert\Omega \rangle\langle\Omega\vert=PP'$ for the projections $P:=\lambda^{-1}X^*X$ of $ x$ and $P':={\lambda}'^{-1}{X'}^*X'$ of $x'$. Precisely as in the proof of the uniqueness of  $P$, $P'$ in \ref{proposition:factorizable-myriad:vi} we see that $[\overline{x'\Omega}]\leq P$ and $[\overline{x\Omega}]\leq P'$. Of course we also have $\vert\Omega \rangle\langle\Omega\vert\leq [\overline{x'\Omega}]$ and $\vert\Omega \rangle\langle\Omega\vert\leq [\overline{x\Omega}]$. But then $\vert\Omega \rangle\langle\Omega\vert=[\overline{x'\Omega}]\vert\Omega \rangle\langle\Omega\vert[\overline{x\Omega}]=[\overline{x'\Omega}]PP'[\overline{x\Omega}]=[\overline{x'\Omega}][\overline{x\Omega}]$.
 

The proof is now complete.
\end{proof}

\begin{definition}
For a factorization $\FFF$ of $\HH$ and $\omega\in \HH$ we say $\omega$ is a factorizable vector of $\FFF$ if $\omega\ne 0$ and  $\Omega:=\frac{\omega}{\Vert\omega\Vert}$ satisfies  any of the equivalent conditions of Proposition~\ref{proposition:factorizable-myriad} for all $x\in \FFF$.
\end{definition}
In terms of the characterization \ref{proposition:factorizable-myriad:v} of Proposition~\ref{proposition:factorizable-myriad}, which is perhaps most basic, the zero vector could potentially  naturally be included in the factorizable ones (since $\vert 0\rangle\langle 0\vert=00$). We stress however that for us only non-zero vectors can be factorizable, which will turn out to be a convenient convention. 
\begin{remark}\label{remark:factorization-typeI}
A factorization can admit  factorizable vectors only if it is of type $\I$ (we see it for instance from Item~\ref{proposition:factorizable-myriad:vi} of Proposition~\ref{proposition:factorizable-myriad}).
\end{remark}

\section{Unital factorizations and their multiplicative vectors}\label{section:unital-factorizations}
Unital factorizations are  factorizations equipped with a distinguished norm one factorizable vector, the unit. Their multiplicative vectors are the factorizable vectors of the underlying factorization having a certain normalization (scalar product equal to one) relative to the distinguished unit. We proceed to detail this. 
\subsection{Unital factorizations}
Let us formally introduce the notion of a ``factorization with unit'' together with some associated terminology.
\begin{definition}
A unital factorization situated on $\HH$ is a pair $(\BBB,\Omega)$ consisting of  a factorization   $\BBB$ of $\HH$  and a norm one factorizable vector $\Omega$ of $\BBB$ (referred to as the unit). In such case:
\begin{itemize}
\item $(\BBB,\Omega)$ is said to be complete if $\BBB$ is so;
\item a unital subfactorization of $(\BBB,\Omega)$ is a unital factorization of the form $(\tilde\BBB,\Omega)$, where $\tilde \BBB$ is a subfactorization of $\BBB$.
\end{itemize}

Given  two unital factorizations $(\BBB_1,\Omega_2)$ and $(\BBB_2,\Omega_2)$ introduced relative to the  Hilbert spaces $\HH_1$ and $\HH_2$, respectively, we: 
\begin{itemize}
\item define their product as the unital factorization $(\BBB_1\otimes \BBB_2,\Omega_1\otimes\Omega_2)$ on $\HH_1\otimes \HH_2$; 
\item say they are isomorphic if there exists  a unitary isomorphism between $\HH_1$ and $\HH_2$ that sends $\Omega_1$ to $\Omega_2$ and carries $\BBB_1$ onto $\BBB_2$. 
\end{itemize}
\end{definition}
\begin{remark}
Thanks to Proposition~\ref{proposition:factorizable-myriad}, specifically the characterization~\ref{proposition:factorizable-myriad:i}, a unital factorization $(\BBB,\Omega)$ is nothing but a factorization $\BBB$ equipped with a norm one vector $\Omega$ satisfying the property  that for all $x\in  \BBB$, $x$ and $x'$ are independent under $\Omega$; and necessarily then $\BBB$ is of type $\I$ (as noted in Remark~\ref{remark:factorization-typeI}). The product of two unital factorizations is indeed a unital factorization, by Proposition~\ref{proposition:raise-independence}.
\end{remark}
 Several examples are now in order. 

\begin{example}\label{example:noise-boolean}
Let $\PP$ be a probability. We denote by $1_\PP$ its basic $\sigma$-field (its domain) and then by $\hat{\PP}$ the collection of all $\PP$-complete sub-$\sigma$-fields of $1_\PP$. (The probability $\PP$ itself need not be complete, however.) The set $\hat\PP$ comes equipped with natural lattice operations of meet $\land=$ intersection and join $\lor=$ generated $\sigma$-field.
A  noise(-type) Boolean algebra under $\PP$ \cite[Definition~1.1]{tsirelson} is a distributive non-empty sublattice  of $\hat\PP$ each member of which admits an independent complement. Here by an independent complement of $x\in \hat\PP$ we mean an $x'\in \hat \PP$ such that $x$ and $x'$ are independent under $\PP$ and $x\lor x'=1_\PP$, whence automatically $x\land x'=\PP^{-1}(\{0,1\})=:0_ \PP$. Let $B$ be a noise Boolean algebra. Interpreting $0_\PP$ as $0$ and $1_\PP$ as $1$, $B$ is a Boolean algebra, independent complements, being unique in $B$, playing the role of complements. For $x\in B$ set $F_x:= 1_{\LLL^2(\PP\vert_x)}\otimes 0_{\LLL^2(\PP\vert_{x'})}\subset 1_{\LLL^2(\PP)}$, which is a type $\I$ factor, the inclusion being up to the natural unitary equivalence $\LLL^2(\PP)=\LLL^2(\PP\vert_x)\otimes \LLL^2(\PP\vert_{x'})$. Put $B^\uparrow:=\{F_x:x\in B\}$ and denote by $\mathbbm{1}$ the constant one function on the sample space of $\PP$. Then $(B^\uparrow,\mathbbm{1})$ is a unital factorization and $(B\ni x\mapsto F_x\in B^\uparrow)$ is an isomorphism of Boolean algebras.
\end{example}

Actually it is not clear whether  this example does not already exhaust unital factorizations up to isomoprhism, which is worth singleing out explicitly as
\begin{question}\label{question}
 (When the underlying Hilbert space is separable) is every unital factorization isomorphic to one coming from a noise Boolean algebra?
\end{question}
We will make no attempt to answer it here.

\begin{example}\label{example:factorizations:cts-Fock} Let $(\FF_s)_{s\in S}$ be a $\nu$-measurable field \cite[Section~II.1.3]{dixmier1981neumann} of a.e.-$\nu$ non-zero Hilbert spaces, $\nu$  atomless. An element $p$ of the measure algebra $\mathsf{b}$  of $\nu$ is identified with the projection of the direct integral \cite[Section~II.1.5]{dixmier1981neumann} $\FF:=\int^\oplus\FF_s\nu(\dd s)$ onto $\int^\oplus_p \FF_s\nu(\dd s)$. For each such $p$ there is a canonical unitary isomorphism between $\Gamma_s(p\FF)\otimes \Gamma_s(p^\perp\FF)$ and $\Gamma_s(\FF)$ which sends $e_{p\FF}(x)\otimes e_{p^\perp \FF}(y)$ to $e_{\FF}(x+y)$  for all $(x,y)\in (p\FF)\times (p^\perp \FF)$; $1_{\Gamma_s(p\FF)}\otimes 0_{\Gamma_s(p^\perp\FF)}$ is sent via this map onto a type $\I$ factor of $\Gamma_s(\FF)$, this factor we shall denote by $F_p$. Then  $\mathrm{c\text{-}Fock}(\int^\oplus \FF_s\nu(\dd s)):=\{F_p:p\in \mathsf{b}\}$ is an atomless factorization of $\Gamma_s(\FF)$, the continuous Fock factorization associated to $\int^\oplus \FF_s\nu(\dd s)$; moreover,  the map $(\mathsf{b}\ni p\mapsto F_p\in \mathrm{c\text{-}Fock}(\int^\oplus \FF_s\nu(\dd s)))$ is an isomorphism of Boolean algebras. Furthermore, denoting by  $1\in \mathbb{C}=\FF^{\otimes 0}\subset \Gamma_s(\FF)$ (up to natural identifications) the vacuum vector, $(\mathrm{c\text{-}Fock}(\int^\oplus \FF_s\nu(\dd s)),1)$ is a unital factorization.
\end{example}
In the discrete case it is better to work with
\begin{example}\label{example:factorizations:discrete-Fock}  Let $(\FF_\alpha)_{\alpha\in \mathfrak{A}}$ be any  collection of Hilbert spaces of dimension at least one and let $\nu$ be a measure on $(\mathfrak{A},2^\mathfrak{A})$  with $\nu(\{\alpha\})\in (0,\infty)$ for all $\alpha\in \mathfrak{A}$. Write $\FF:= \oplus_{\alpha\in \mathfrak{A}}\nu(\{\alpha\})\FF_\alpha$ and $\FF_A:= \oplus_{\alpha\in A}\nu(\{\alpha\})\FF_\alpha$ for $A\subset\mathfrak{A}$. Put $\gamma_s(\FF):=\oplus_{F\in (2^\mathfrak{A})_{\mathrm{fin}}} \left(\left(\prod_{f\in F}\nu(\{f\})\right)\otimes_{f\in F}\FF_f\right)$.  For $x\in \FF$ denote $e_{\FF}(x):=\oplus_{F\in (2^\mathfrak{A})_{\mathrm{fin}}}\otimes (x\vert_F)$, the ``exponential'' (geometric?) vector associated to $x$. Then for $A\in 2^\mathfrak{A}$ we have the decomposition $\gamma_s(\FF)=\gamma_s(\FF_A)\otimes\gamma_s(\FF_{\mathfrak{A}\backslash A})$ up to the natural unitary isomorphism which carries $e_{\FF_A}(x)\otimes e_{\FF_{\mathfrak{A}\backslash A}}(y)$ to $e_{\FF}(x+y)$  for all $(x,y)\in  \FF_A\times\FF_{\mathfrak{A}\backslash A}$; this unitary isomorphism transfers $1_{\gamma_s(\FF_A)}\otimes 0_{\gamma_s(\FF_{\mathfrak{A}\backslash A})}$ onto a  type $\I$ factor of $\gamma_s(\FF)$, which we  denote by $F_A$.  Then $\mathrm{d\text{-}Fock}(\oplus_{\alpha\in \mathfrak{A}}\nu(\{\alpha\})\FF_\alpha):=\{F_A:A\in 2^\mathfrak{A}\}$ is a factorization of $\gamma_s(\FF)$, the  discrete Fock factorization associated to $\oplus_{\alpha\in \mathfrak{A}}\nu(\{\alpha\})\FF_\alpha$; moreover $(2^\mathfrak{A}\ni A\mapsto F_A\in\mathrm{d\text{-}Fock}(\oplus_{\alpha\in \mathfrak{A}}\nu(\{\alpha\})\FF_\alpha))$ is an isomorphism of Boolean algebras.  Furthermore, $(\mathrm{d\text{-}Fock}(\oplus_{\alpha\in \mathfrak{A}}\nu(\{\alpha\})\FF_\alpha),1)$, where $1\in \mathbb{C}=\otimes_{f\in \emptyset}\FF_f\subset \gamma_s(\FF)$ (up to natural identifications) plays the role of the ``vacuum'' vector, is a unital factorization.
\end{example}


\begin{example}\label{example:discrete-factorization}
Let $(\JJ_\beta)_{\beta\in \mathfrak{B}}$ be a  family of  Hilbert spaces of dimension at least two. For each $\beta\in \mathfrak{B}$ there is given also a norm one vector $e_\beta\in \JJ_\beta$. We  form the incomplete tensor product $\JJ:=\otimes_{\beta\in \mathfrak{B}}\JJ_\beta$ relative to the stabilising family $e=(e_\beta)_{\beta\in \mathfrak{B}}$ \cite[Section~3]{araki-woods} \cite[Exercise~15.10]{Parthasarathy}. (For the ``uninitiated'' reader let us recall: we have a map $\otimes$ from $L:=\{g\in \prod_{\beta\in \mathfrak{B}}\JJ_\beta:g_\beta=e_\beta\text{ for all but finitely many }\beta\in \mathfrak{B}\}$ onto a total subset of $ \JJ$ satisfying $\langle \otimes g,\otimes f\rangle=\prod_{\beta\in \mathfrak{B}}\langle g_\beta,f_\beta\rangle$ for  $\{f,g\}\subset L$, which determines the pair $(\JJ,\otimes)$ uniquely up to unitary equivalence.) For $B\in 2^\mathfrak{B}$, $\JJ= (\otimes_{\beta\in B}\JJ_\beta)\otimes (\otimes_{\beta\in \mathfrak{B}\backslash B}\JJ_\beta)=:\JJ_B\otimes \JJ_{\mathfrak{B}\backslash B}$ up to the natural unitary isomorphism taking $\otimes (g\cup f)$ to $(\otimes g)\otimes (\otimes f)$ for all $(g,f)\in (\prod_{\beta\in A}\JJ_\beta)\times (\prod_{\beta\in \mathfrak{B}\backslash A}\JJ_\beta)$, $\JJ_B$ and $\JJ_{\mathfrak{B}\backslash B}$ being defined here relative to the stabilising families $e\vert_B$ and $e\vert_{\mathfrak{B}\backslash B}$; $1_{\JJ_B} \otimes 0_{ \JJ_{\mathfrak{B}\backslash B}}$ is carried by this isomorphism onto a type $\I$ factor of $\JJ$ that we denote by $F_B$. Then $\prod(\otimes_{\beta\in \mathfrak{B}}\JJ_\beta^{e_\beta}):=\{F_B:B\in 2^\mathfrak{B}\}$ is a factorization of $\JJ$; moreover, the map $(2^\mathfrak{B}\ni B\mapsto F_B\in \prod(\otimes_{\beta\in \mathfrak{B}}\JJ_\beta^{e_\beta}))$ is an isomorphism of Boolean algebras. Furthermore, $(\prod(\otimes_{\beta\in \mathfrak{B}}\JJ_\beta^{e_\beta}),\otimes_{\beta\in \mathfrak{B}}e_\beta)$ is a unital factorization. 
\end{example}
The unital factorizations of Examples~\ref{example:factorizations:discrete-Fock}  and~\ref{example:discrete-factorization} are just a rewriting of each other in the sense that when $\mathfrak{A}=\mathfrak{B}$ and $\JJ_\beta=(\mathbb{C}e_\beta)\oplus \nu(\{\beta\})\FF_\beta$ for $\beta\in \mathfrak{B}$, then $(\mathrm{d\text{-}Fock}(\oplus_{\alpha\in \mathfrak{A}}\nu(\{\alpha\})\FF_\alpha),1)$ and $(\prod(\otimes_{\beta\in \mathfrak{B}}\JJ_\beta^{e_\beta}),\otimes_{\beta\in \mathfrak{B}}e_\beta)$ are isomorphic by  the natural unitary isomorphism which results from ``opening the brackets'' in $\otimes_{\beta\in \mathfrak{B}}\JJ_\beta=\otimes_{\beta\in \mathfrak{B}}((\mathbb{C}e_\beta)\oplus\nu(\{\beta\}) \FF_\beta)$.\label{page:obseervation}

\begin{definition}\label{definition:noise-factorization}
Let $(T,\TT,\nu)$ be a $\sigma$-finite  standard  measure space and let  $(\FF_s)_{s\in S}$ be a measurable field of separable, a.e.-$\nu$ non-zero Hilbert spaces.  We define $\mathrm{Fock}(\int^\oplus\FF_s\nu(\dd s))$ as the product of the  factorizations of Examples~\ref{example:factorizations:cts-Fock} and~\ref{example:factorizations:discrete-Fock} taking respectively $\int^\oplus \FF_s\nu_c(\dd s)$ and $\oplus_{\alpha\in \at(\nu)}\nu(\{\alpha\})\FF_\alpha$ as the input data. This  factorization of the separable Hilbert space $\boldsymbol{\Upgamma}_s(\int^\oplus\FF_s\nu(\dd s)):=\Gamma_s(\int^\oplus\FF_s\nu_c(\dd s))\otimes \gamma_s(\oplus_{\alpha\in\at(\nu)}\nu(\{\alpha\})\FF_\alpha)$ is said to be Fock. Effecting the product of the unital factorizations, i.e. taking along the units, we get a Fock unital factorization  $(\mathrm{Fock}(\int^\oplus\FF_s\nu(\dd s)),1)$, where we have taken the liberty of denoting the unit with $1$ yet again. For $u\in \int^\oplus \FF_s\nu(\dd s)$ we write $e_{\int ^\oplus\FF_s\nu(\dd s)}(u):=e_{\int^\oplus\FF_s\nu_c(\dd s)}(u\vert_{T\backslash \at(\nu)})\otimes e_{\oplus_{\alpha\in\at(\nu)}\nu(\{\alpha\})\FF_\alpha}(u\vert_{\at(\nu)})$ for the exponential (perhaps we should call it the exponential-geometric) vector associated to $u$,  while for $W\in \TT$ we  set $F_{\int ^\oplus\FF_s\nu(\dd s)}(W):=F_{W\backslash \at(\nu)}\otimes F_{W\cap\at(\nu)}$. A  unital factorization is said to be of Fock type if it is isomorphic to a unital subfactorization of a   Fock unital factorization.
\end{definition}
In the context of Definition~\ref{definition:noise-factorization}
\begin{equation}\label{equation:scalar-product-in-exponential-Fock-space}
\left\langle e_{\int ^\oplus\FF_s\nu(\dd s)}(u),e_{\int ^\oplus\FF_s\nu(\dd s)}(v)\right\rangle=e^{\int\langle u(t),v(t)\rangle\nu_c(\dd t)}\prod_{t\in T}(1+\nu(\{t\})\langle u(t),v(t)\rangle),\quad \{u,v\}\subset  \int^\oplus \FF_s\nu(\dd s).
\end{equation}
 A Fock factorization  is complete; as a matter of fact, the isomorphisms of Boolean algebras noted in Examples~\ref{example:factorizations:cts-Fock}-\ref{example:discrete-factorization} preserve arbitrary (not just finite) joins (and hence also meets) \cite[Eq.~(3.30) and p.~ 202]{araki-woods}. Every Fock factorization is also type $\I$.

\begin{center}
\fbox{Henceforth $(\BBB,\Omega)$ is a unital factorization based on the  Hilbert space $\HH$.}
\end{center}

\begin{definition}
For $x\in \hat\HH$ we set $\HH_x:=\overline{x\Omega}$ and then $\phi_x:=[\HH_x]$. We also put $\Phi:=\{\phi_x:x\in \BBB\}$.
\end{definition}
Trivially we identify $\HH_{0_\HH}=\mathbb{C}\Omega$, $\phi_{0_\HH}=\vert\Omega\rangle\langle \Omega\vert$ 
and $\HH_{1_\HH}=\HH$, $\phi_{1_\HH}=\mathbf{1}_\HH$ 
in parallel. If $(x_n)_{n\in \mathbb{N}}$ is a sequence in $\hat\HH$ that is $\uparrow$ (resp. $\downarrow$) to an $x\in\hat\HH$, then $\phi_{x_n}\uparrow \phi_x$ (resp. $\phi_{x_n}$ is $\downarrow$ and the limit is $\geq \phi_x$) as $n\to\infty$.

\begin{example}\label{example:noise-Boolean-bis}
Relative to the unital factorization $(B^\uparrow,\mathbbm{1})$ of Example~\ref{example:noise-boolean}, for $x\in B$, $\HH_{F_x}=\LLL^2(\PP\vert_x)$ and $\phi_{F_x}=\PP[\cdot\vert x]$, the conditional expectation w.r.t. $x$.
\end{example}

On account of Proposition~\ref{proposition:factorizable-myriad}\ref{proposition:factorizable-myriad:iii}, for  each $x\in \BBB$  there exists a unique unitary isomorphism between $\HH$ and $\HH_x\otimes \HH_{x'}$ that 
\begin{itemize}
\item sends $\Omega$ to $\Omega\otimes \Omega$ and sends $XX'$ to  $X\vert_{\HH_x}\otimes X'\vert_{\HH_{x'}}$ for all $X\in x$ and all  $X'\in x'$,
\item therefore $\HH_x$ onto $\HH_x\otimes (\mathbb{C}\Omega)$ and $\HH_{x'}$ onto $(\mathbb{C}\Omega)\otimes \HH_{x'}$, 
\item whence also $x$ onto $1_{\HH_x}\otimes 0_ {\HH_{x'}}$ and $x'$ onto $0_{\HH_{x}}\otimes 1_{\HH_{x'}}$;
\end{itemize}
\begin{equation}\label{identification}
\text{\fbox{henceforth we silently identify $\HH$ with $\HH_x\otimes \HH_{x'}$ as per above.}}
\end{equation}
 Some  immediate consequences  of the preceding observation ensue. 

%

\begin{definition}
Let $x\in \BBB$. We set $\BBB_x:=\BBB\cap 2^x$, $y\vert_{\HH_x}:=\{Y\vert_{\HH_x}:Y\in y\}$ for $y\in \BBB_x$ and then $\BBB_x\vert_{\HH_x}:=\{y\vert_{\HH_x}:y\in \BBB_x\}$. 
\end{definition}

\begin{proposition}\label{proposition:restriction-of-noise-factorization}
For $x\in \BBB$, $(\BBB_x\vert_{\HH_x},\Omega)$ is a unital factorization based on the Hilbert space $\HH_x$ and the map $(\BBB\ni y\mapsto (y\land x)\vert_{\HH_x}\in \BBB_x\vert_{\HH_x})$ is a homomorphism of Boolean algebras. 
\end{proposition}
\begin{proof}
Notice that $\HH_x\ne \{0\}$ and that $\BBB_x\vert_{\HH_x}\subset \widehat{\HH_x}$. \eqref{identification} identifies a $y=(y\land x)\lor (y\land x')$ from $\BBB$ with $(y\land x)\vert_{\HH_x}\otimes (y\land x')\vert_{\HH_{x'}}$, hence $y'$ with $((y\land x)\vert_{\HH_x})'\otimes ( (y\land x')\vert_{\HH_{x'}})'$; thus $y'\land x$ with $(((y\land x)\vert_{\HH_x})'\otimes  ( (y\land x')\vert_{\HH_{x'}})')\land (1_{\HH_x}\otimes 0_ {\HH_{x'}})=((y\land x)\vert_{\HH_x})'\otimes 0_{\HH_{x'}}$ on the one hand, but also with $((y'\land x)\vert_{\HH_x})\otimes 0_{\HH_{x'}}$ on the other. This shows that the indicated map preserves the commutant operation and clearly it does also joins; besides, it sends $0_\HH$ to $0_{\HH_x}$ and $1_\HH$ to $1_{\HH_x}$. Now the claim follows easily.
\end{proof}
Thus, for $x\in \BBB$, $(\BBB,\Omega)$ is (isomorphic to) the product of  the ``restrictions'' $(\BBB_x\vert_{\HH_x},\Omega)$ and $(\BBB_{x'}\vert_{\HH_{x'}},\Omega)$, which may be referred to as the ``local product decomposition'' of $(\BBB,\Omega)$.

 \begin{proposition}\label{proposition:orthogonal}
 If $\{x,y\}\subset \BBB$ and $x\land y=0_\HH$, then $\HH_x$ and $\HH_y\ominus (\mathbb{C}\Omega)$ are orthogonal.
 \end{proposition}
 \begin{proof}
 We may and do assume $y=x'$. Then for $\xi\in \HH_x$ and $\xi'\in \HH_{x'}$, $\langle \xi,\xi'-[\mathbb{C}\Omega]\xi'\rangle=\langle \xi\otimes \Omega,\Omega\otimes \xi'-\langle \Omega,\xi'\rangle\Omega\otimes\Omega\rangle=\langle \xi,\Omega\rangle\langle \Omega,\xi'\rangle-\langle \Omega,\xi'\rangle\langle \xi,\Omega\rangle\langle\Omega,\Omega\rangle=0$ ($\because$ $\Omega$ is unit),  yielding the claim.
 \end{proof}

\begin{proposition}\label{rmk:parition-of-unity-tensor-decomposition}
For every partition of unity $P$  of  $\BBB$    there exists a unique unitary isomorphism between $\HH$ and $\otimes_{p\in P}\HH_p$ that sends $\Omega$ to $\Omega^{\otimes P}$ and  $\prod_{p\in P}X_p$ to  $\otimes_{p\in P}X\vert_{\HH_p}$ for all choice functions $X$ on $P$.
\end{proposition}
\begin{proof}
An inductive argument utilising Proposition~\ref{proposition:restriction-of-noise-factorization} for existence. Uniqueness can be argued directly: the stipulation determines the unitary isomorphism uniquely on a total set. Alternatively, both existence and uniqueness are seen at once following the line of the proof of Proposition~\ref{quantum-independence}, \ref{quantum-independence:A} $\Rightarrow$ \ref{quantum-independence:C}. 
\end{proof}

%

Via the unitary isomorphism of Proposition~\ref{rmk:parition-of-unity-tensor-decomposition} 
 we
\begin{equation}\label{identification-another}
\text{\fbox{henceforth silently identify $\HH$ with  $\otimes_{p\in P}\HH_p$,}}
\end{equation}
hence  an $x\in \BBB$ with $\otimes_{p\in P}(p\land x)\vert_{\HH_p}$, $\HH_x$ with $\otimes_{p\in P}\HH_{p\land x}$ and  $\phi_x$ with $\otimes_{p\in P}\phi_{p\land x}\vert_{\HH_p}$. Naturally, for any finite $Q\subset \BBB\backslash \{0_\HH\}$ satisfying $x\land y=0_\HH$ for $x\ne y$ from $Q$, we also
\begin{equation}\label{identification-another'}
\text{\fbox{henceforth  silently identify $\HH_{\lor Q}$ with $\otimes_{q\in Q}\HH_q$}}
\end{equation}
(formally through utilising first Proposition~\ref{proposition:restriction-of-noise-factorization} with $x=\lor Q$ and then Proposition~\ref{rmk:parition-of-unity-tensor-decomposition} for this restricted unital factorization).

\begin{remark}\label{remark:connection}
In virtue of \eqref{identification-another}, unital factorizations may be viewed as a certain abstraction of  continuous products of pointed Hilbert spaces \cite[Definition~6d6]{tsirelson-nonclassical}/continuous tensor product systems of Hilbert spaces with a unit \cite[Definitions~3.1 and~3.6]{liebscher}. In this sense Question~\ref{question} corresponds to \cite[Question~6d7]{tsirelson-nonclassical}.  Directly from the definition they are also an abstraction of continuous $\mathbb{C}$-expected Bernoulli shifts \cite[Definition~3.1.2]{HKK} in the case of a ``commuting past and future'' \cite[Subsection~4.3]{HKK}.
\end{remark}

\begin{corollary}\label{corollary:factorizable-and-partition-of-unity}
For all partitions of unity $P$ of $\BBB$, $\vert \Omega\rangle\langle \Omega\vert=\prod_{p\in P}X_p$ for some choice function $X$ on $P$ having as $X_p$ a minimal non-zero projection of $p$ for all $p\in P$, while 
\begin{equation*}
\left\langle \Omega,\prod_{p\in P}X_p\Omega\right\rangle=\prod_{p\in P}\langle \Omega,X_p\Omega\rangle
\end{equation*}
for all choice functions $X$ on $P$.
\end{corollary}
\begin{proof}
From Proposition~\ref{rmk:parition-of-unity-tensor-decomposition} we get $\vert \Omega\rangle\langle \Omega\vert=\otimes_{p\in P}\vert \Omega\rangle\langle \Omega\vert$ up to the identification \eqref{identification-another}. 
\end{proof}
Corollary~\ref{corollary:factorizable-and-partition-of-unity} brings our definition of factorizability fully in line with \cite[Definitions~2.3 and 5.2]{araki-woods}, cf. also \cite[Lemma~2.6]{araki-woods} \cite[p.~87]{vershik-tsirelson}. A key observation concerning the family $\Phi$  is the content of

\begin{proposition}\label{proposition:commuting-projections}
 The family $\Phi$ of projections in $\HH$ is commuting, moreover,
\begin{equation}\label{phi-intersection}
\phi_x\phi_y=\phi_y\phi_x=\phi_x\land\phi_y=\phi_{x\land y},\quad \{x,y\}\subset \BBB.
\end{equation}
The map $(\BBB\ni x \mapsto \phi_x)$ is injective.
\end{proposition}
 \begin{proof}
 Immediate from \eqref{identification-another}.
 \end{proof}

\subsection{Multiplicative vectors}The following concept will feature prominently in what follows.
\begin{definition}\label{definition:multiplicative-vector}
A multiplicative vector  is an element $\xi$ of $\HH$ such that $\langle \Omega,\xi\rangle=1$ and such that for all $x\in \BBB$, $\xi=\phi_x(\xi)\otimes \phi_{x'}(\xi)\in \HH_x\otimes \HH_{x'}= \HH$. 
\end{definition}
$\Omega$ is always a multiplicative vector.
\begin{example}\label{example:noise-boolean-contd}
In the context of Example~\ref{example:noise-boolean} a multiplicative vector of $(B^\uparrow,\mathbbm{1})$ is nothing other than a square-integrable multiplicative integral of $B$, i.e. a $\xi\in \LLL^2(\PP)$ such that $\PP[\xi]=1$ and $\xi=\PP[\xi\vert x]\PP[\xi\vert x']$ for all $x\in B$.
\end{example}
 If $\xi$ is a multiplicative vector, then for all $x\in \BBB$ so too is $\phi_x(\xi)$ and (consequently) $\xi=\otimes_{p\in P}\phi_p(x)\in \otimes_{ p\in P}\HH_p=\HH$ for all partitions of unity $P$ of $\BBB$.

The definition of a multiplicative vector involves the family of projections $(\phi_x)_{x\in \BBB}$, and the decomposition \eqref{identification}, but we can give a more direct characterization, which exposes the fundamental connection between factorizable and multiplicative vectors.
\begin{proposition}\label{proposition:multiplicative-factorizable}
The multiplicative vectors of $(\BBB,\Omega)$ are precisely the factorizable vectors  $\xi$ of $\BBB$ for which $\langle \Omega,\xi\rangle=1$. 
\end{proposition}
\begin{proof}
If $\xi$ is a multiplicative vector of $(\BBB,\Omega)$, then by definition $\langle\Omega,\xi\rangle=1$, while for all $x\in \BBB$, up to the identification \eqref{identification},
\begin{equation*}
\vert \xi\rangle\langle \xi\vert=(\vert \phi_x(\xi)\rangle\langle \phi_x(\xi)\vert)\otimes (\vert \phi_{x'}(\xi)\rangle\langle \phi_ {x'}(\xi)\vert)=\left((\vert \phi_x(\xi)\rangle\langle \phi_x(\xi)\vert)\otimes \mathbf{1}_{\HH_{x'}}\right)\left(\mathbf{1}_{\HH_x}\otimes  (\vert \phi_{x'}(\xi)\rangle\langle \phi_ {x'}(\xi)\vert)\right);
\end{equation*}
together with $(\vert \phi_x(\xi)\rangle\langle \phi_x(\xi)\vert)\otimes \mathbf{1}_{\HH_{x'}}\in 1_{\HH_x}\otimes 0_{\HH_{x'}}=x$ and $\mathbf{1}_{\HH_x}\otimes  (\vert \phi_{x'}(\xi)\rangle\langle \phi_ {x'}(\xi)\vert)\in 0_{\HH_x}\otimes 1_{\HH_{x'}}=x'$ it means that $\xi$ is a factorizable vector of $\BBB$. 

Now suppose $\xi$ is a factorizable vector of $\BBB$ and $\langle \Omega,\xi\rangle=1$. Pick any $x\in\BBB$. By definition $\vert \xi\rangle\langle \xi\vert=YY'$ for some $Y\in x$ and $Y'\in x'$. 
The condition $\langle \Omega,\xi\rangle=1$ entails 
\begin{equation}\label{fact-mult:0}
\xi=\vert \xi\rangle\langle \xi\vert\Omega=YY'\Omega
\end{equation} 
and then also 
\begin{equation}\label{fact-mult:i}
\langle  \Omega,Y\Omega\rangle\langle  \Omega,Y'\Omega\rangle=\langle\Omega, YY'\Omega\rangle=\langle\Omega,\xi\rangle=1
\end{equation} 
by the independence of $x$ and $x'$ under $\Omega$. Next, for all $X\in x$, 
\begin{equation*}
\langle X\Omega,\phi_x(\xi)\rangle=\langle X\Omega,\xi\rangle=\langle X\Omega,YY'\Omega\rangle=\langle \Omega,X^*Y\Omega\rangle\langle  \Omega,Y'\Omega\rangle=\langle X\Omega,\langle  \Omega,Y'\Omega\rangle Y\Omega\rangle;
 \end{equation*}
 $X\in x$ being arbitrary we deduce (from $\{\phi_x(\xi),Y\Omega\}\subset \HH_x$ and the totality of $x\Omega$ in $\HH_x$) that 
 \begin{equation}\label{fact-mult:ii}\phi_x(\xi)=\langle  \Omega,Y'\Omega\rangle Y\Omega.
 \end{equation} By the same token we obtain 
 \begin{equation}\label{fact-mult:iii}
 \phi_{x'}(\xi)=\langle  \Omega,Y\Omega\rangle Y'\Omega. 
 \end{equation}
 Combining \eqref{fact-mult:0}-\eqref{fact-mult:iii} we may compute, up to the identification \eqref{identification},
 \begin{equation*}
 \phi_x(\xi) \otimes \phi_{x'}(\xi)=\langle  \Omega,Y\Omega\rangle{\langle  \Omega,Y'\Omega\rangle}(Y\Omega)\otimes (Y'\Omega)=\langle  \Omega,Y\Omega\rangle{\langle  \Omega,Y'\Omega\rangle}YY'\Omega=\langle  \Omega,Y\Omega\rangle{\langle  \Omega,Y'\Omega\rangle}\xi=\xi,
 \end{equation*}
 rendering $\xi$ a multiplicative vector of $(\BBB,\Omega)$.
\end{proof}


\section{Spectrum of a unital factorization}\label{section:spectrum}
The concepts and results of this section for the commutative landscape can be found in large part in \cite{tsirelson};  in \cite{vidmar-noise} in the complementary part.

\begin{definition}
The von Neumann algebra generated on $\HH$ by $\Phi$ is denoted $\AA$. 
\end{definition}
By Proposition~\ref{proposition:commuting-projections} $\Phi$ is commuting, hence  so too is $\AA$. 
\begin{center}
\fbox{Henceforth $\HH$ is separable.}
\end{center}
Thus $\AA$ is an abelian von Neumann algebra acting on the separable Hilbert space $\HH$.

Spectrally resolving \cite[A.84]{dixmier-c-star} $\AA$ we get  a standard $\sigma$-finite measure space $(S,\Sigma,\mu)$, a measurable field of separable a.e.-$\mu$ non-zero Hilbert spaces $(\GG_s)_{s\in S}$ and a unitary isomorphism $\Psi:\HH\to \int^\oplus \GG_s\mu(\dd s)$ that carries $\AA$ onto the algebra of diagonalizable operators \cite[A.80]{dixmier-c-star} (which we identify naturally with  $\LLL^\infty(\mu)$) as a $^*$-isomorphism. This $^*$-isomorphism we shall denote by $\alpha$. Thus  $\alpha$ maps $\AA$ onto $ \LLL^\infty(\mu)$ bijectively. We call $(\mu,\Psi)$ briefly a spectrum for $(\BBB,\Omega)$.  One could insist on the measure $\mu$ being finite, but it is somewhat more forgiving to allow a $\sigma$-finite space.


\subsection{Spectral subspaces, spectral sets, spectral measures}
If $A\in \Sigma$, then $\int^\oplus_A \GG_s\mu(\dd s)$ is a closed linear subspace of $\int^\oplus \GG_s\mu(\dd s)$ that $\Psi$ pulls back onto a closed linear subspace of $\HH$ that we shall denote  $\HHH(A)$. 

For $x\in \BBB$, 
$\alpha(\phi_x)$, being a projection from $\LLL^\infty(\mu)$, is (multiplication by) $\mathbbm{1}_E$ for some a.e.-$\mu$ unique $E\in \Sigma$ that we denote by $S_x$ and call the spectral set of $x$; accordingly $\HHH(S_x)=\phi_x(\HH)=\HH_x$. 
Due to \eqref{phi-intersection},
\begin{equation}\label{spectral-sets:pi-system}
S_{x\land y}=S_x\cap S_y\text{ a.e.-$\mu$},\quad  \{x,y\}\subset \BBB.
\end{equation}
The $^*$-isomorphism $\alpha$ being onto $\LLL^\infty(\mu)$ we see that (a countable subset of) the collection $\{S_x:x\in \BBB\}$ generates $\Sigma$ up to $\mu$-trivial sets (a countable subset suffices, because $\mu$ is standard).  Since $\phi_{0_\HH}=\vert\Omega\rangle\langle\Omega\vert$ is a minimal non-zero projection of $\AA$, $S_{0_\HH}$ is an atom of $\mu$: $
S_{0_\HH}=\{\emptyset_S\}\text{ a.e.-$\mu$}$
for a unique $\emptyset_S\in S$ that is charged by $\mu$. In this notation $\HHH(\{\emptyset_S\})=\mathbb{C}\Omega$. 

For $g\in \HH$ we write $\mu_g:=\Vert \Psi(g)\Vert^2\cdot \mu$ (wherein the norm is taken pointwise on $S$, of course), which is a finite measure on $(S,\Sigma)$ with total mass $\Vert g\Vert^2$ satisfying $\mu_g(S_x)=\Vert \phi_x(g)\Vert^2$, $x\in \BBB$. More generally, for further $h\in \HH$, $\mu_{h,g}:=\langle \Psi(h),\Psi(g)\rangle\cdot\mu$ (the scalar product having been taken, again, pointwise on $S$) is a complex measure on $(S,\Sigma)$ verifying $\mu_{h,g}(S_x)=\langle \phi_x(h),\phi_x(g)\rangle=\langle \phi_x(h),g\rangle=\langle h,\phi_x(g)\rangle$, $x\in \BBB$.
\subsection{The counting map}

By $\mathfrak{F}$ let us denote the collection of all finite Boolean subalgebras of $\BBB$. Any $b\in \mathfrak{F}$ has its collection of atoms $\mathrm{at}(b)$, which is a partition of unity of $\BBB$. We introduce 
\begin{equation}\label{eq:Kb}
K_b:=\sum_{x \in \at(b)}\mathbbm{1}_{S\backslash S_{x'}}.
\end{equation}
 Due to \eqref{spectral-sets:pi-system}, for $\mu$-a.e. $s$, $K_b(s)$ is the number of atoms of $b$ comprising (which is to say, whose join is) the minimal element $x\in b$ such that $s\in S_x$, this minimal element $x$ let us denote by $\underline{b}(s)$. Thus $\underline{b}:S\to b$; and, for $\mu$-a.e. $s$, 
 \begin{equation}\label{eq:ad-def-K}
\forall x\in b: \text{$s\in S_x$ iff $x\supset \underline{b}(s)$.}
 \end{equation}
 If $\{b_1,b_2\}\subset \mathfrak{F}$ and $b_1\subset b_2$ then $\underline{b_2}\subset \underline{b_1}$ a.e.-$\mu$, accordingly $K_{b_2}\geq K_{b_1}$ a.e.-$\mu$. The family $(K_b)_{b\in \mathfrak{F}}$ is therefore directed upwards in the sense that for $\{b_1,b_2\}\subset \mathfrak{F}$ there is $b\in \mathfrak{F}$ such that $K_b\geq K_{b_1}\lor K_{b_2}$ a.e.-$\mu$ (namely, we can take $b=b_1\lor b_2$). We define 
 \begin{equation}\label{def:counting-map}
 K:=\text{$\mu$-$\esssup_{b\in \mathfrak{F}}$}\, K_b,
 \end{equation} which is thus a $\Sigma$-measurable map valued in $\mathbb{N}_0\cup \{\infty\}$, that we shall refer to as the counting map. Of course $\{K=0\}=S_{0_\HH}=\{\emptyset_S\}$ a.e.-$\mu$. 
 By the very definition \eqref{def:counting-map} of $K$, the fact that $(K_b)_{b\in \mathfrak{F}}$ is directed upwards and the properties of an essential supremum there exists a $\uparrow$ sequence $(b_n)_{n\in \mathbb{N}}$ in $\mathfrak{F}$ such that 
\begin{equation}\label{xK-limit}
K=\text{$\uparrow$-$\lim_{n\to\infty}$}K_{b_n}\text{ a.e.-$\mu$};
\end{equation}
any such sequence shall be said to be exhausting and we may ask of it for $b_1$ to be whichever member of $\mathfrak{F}$ that we so please.

We denote by $\mu^{(1)}$ the  restriction of $\mu$ to $\{K=1\}$.
 
 \subsection{Spectral projections and spectral independence}\label{subsection:spectral-everything}
 Let $x\in\BBB$. Recalling Proposition~\ref{proposition:restriction-of-noise-factorization},  in the obvious notation, $(\mu_x,\Psi_x):=(\mu\vert_{S_x},\Psi\vert_{\HH_x}\vert_{S_x})$ is a spectrum for $(\BBB_x\vert_{\HH_x},\Omega)$, i.e. $\Psi_x:\HH_x\to\int^\oplus \GG_s\mu_x(\dd s)$ is a spectral resolution of the commutative von Neumann algebra $\AA_x={\Phi_x}''$ generated by $\Phi_x=\{\phi_u\vert_{\HH_x}:u\in \BBB_x\}$ on $\HH_x$.  

Still $x\in \BBB$. We have $\Phi=\Phi_x\dot{\otimes }\Phi_{x'}:=\{\phi_u\vert_{\HH_x}\otimes \phi_v\vert_{\HH_{x'}}:(u,v)\in \BBB_x\times\BBB_{x'}\}$  and so $\AA=\AA_x\otimes \AA_{x'}$. Therefore, a spectrum for $(\BBB,\Omega)$ is given by the pair $(\mu_x\times \mu_{x'},\Psi_x\otimes \Psi_{x'})$, $\Psi_x\otimes \Psi_{x'}$ mapping $\HH=\HH_x\otimes \HH_{x'}$  onto the direct integral of tensor products \cite[Proposition~II.1.10]{dixmier1981neumann} $\int^\oplus \GG_s\otimes \GG_{s'}(\mu_x\times\mu_{x'})(\dd (s,s'))=\int ^\oplus\GG_s\mu_x(\dd s)\otimes \int^\oplus \GG_{s'}\mu_{x'}(\dd s')$ (equality up to the natural identification). Furthermore, for $(u,v)\in \BBB_x\times \BBB_{x'}$  
\begin{equation}\label{eq:spectral-in-product}
\text{the spectral set associated to $u\lor v$ relative to $(\mu_x\times \mu_{x'},\Psi_x\otimes \Psi_{x'})$  is equal to $S_u\times S_v$ a.e.-$(\mu_x\times \mu_{x'})$.}
\end{equation} 
But we also continue to have the given standing spectrum $(\mu,\Psi)$ for $(\BBB,\Omega)$. In consequence, by the ``uniqueness of spectral resolutions of abelian von Neumann algebras'' \cite[A.85]{dixmier-c-star}, there exists a 
\begin{equation}\label{eq:spectral-resolve-mod0}
\text{mod-$0$ isomorphism $\psi$ between  $\mu_{x}\times \mu_{x'}$ and a $\sigma$-finite measure $\mu'$ equivalent to $\mu$,}
\end{equation}
 together with a 
\begin{equation}\label{eq:spectral-resolve-spaces}
 \text{$\psi$-isomorphism \cite[A.70]{dixmier-c-star} $\zeta$ between  $(\GG_s\otimes \GG_{s'})_{(s,s')\in S_x\times S_{x'}}$ and $(\GG_s)_{s\in S}$ (modulo negligible sets),}
 \end{equation}
 such that  $\Psi\circ (\Psi_x\otimes\Psi_{x'})^{-1}$ is  the composition of $\int^\oplus \zeta(s,s')(\mu_x\times \mu_{x'})(\dd(s,s'))$ and of the canonical unitary isomorphism \cite[Remark on p.~170]{dixmier1981neumann} between $\int^\oplus \GG_s\mu'(\dd s)$ and $\int^\oplus \GG_s\mu(\dd s)$. In particular, thanks to \eqref{eq:spectral-in-product}, the mod-$0$ isomorphism $\psi$ verifies
  \begin{equation}\label{eq:spectral-resolve-sets}
\text{$\psi(S_u\times S_v)=S_{u\lor v}$ a.e.-$\mu$ for all $(u,v)\in \BBB_x\times \BBB_{x'}$}.
  \end{equation}

By \eqref{eq:spectral-resolve-mod0} and \eqref{eq:spectral-resolve-sets}, considering $(\pr_{S_x},\emptyset_S):S_x\times S_{x'}\to S_x\times S_{x'}$, where $\pr_{S_x}$ is the canonical projection of $S_x\times S_{x'}$ onto its first factor $S_x$ and where we treat $\emptyset_S$ as the constant map on $S_x\times S_{x'}$, we see that there exists an a.e.-$\mu$ unique $\Sigma/\Sigma$-measurable map $\pr_x$ such that $\pr_x^{-1}(F)$ is $\mu$-negligible for $\mu$-negligible $F$ and such that 
 \begin{equation*}
 \pr_x^{-1}(S_y)=S_{y\lor x'}\text{ a.e.-$\mu$ for all $y\in \BBB$ (equivalently, all $y\in \BBB_x$)}.
 \end{equation*}
Here for the uniqueness one observes that, by standardness, a countable subcollection of $\{S_u:u  \in \BBB_x\}$ separates the points of a $\mu_x$-conegligible set. Let us call $\pr_x$ the spectral projection onto $x$; plainly it takes its values in $S_x$ a.e.-$\mu$ and the $\mu$-complete  $\sigma$-field which it generates is $\pr_x^{-1}(\Sigma)\lor \mathcal{N}_\mu=\sigma(\{S_{u\lor x'}:u\in \BBB_x\}) \lor \mathcal{N}_\mu=:\Sigma_x$, $\mathcal{N}_\mu:=\{E\in \Sigma:\mu(E)=0\text{ or }\mu(S\backslash E)=0\}$ being the $\mu$-trivial sets.   

Incidentally, \eqref{eq:spectral-resolve-spaces} reveals a kind of ``local tensor product form'' of the spectral resolution reflecting the local product form of $(\BBB,\Omega)$ noted just after Proposition~\ref{proposition:restriction-of-noise-factorization}

\begin{definition}
 A probability $\nu$  on $(S,\Sigma)$ shall be called a spectral independence probability if $\nu$ is absolutely continuous w.r.t. $\mu$, $\nu(\{\emptyset_S\})>0$ and $\pr_x$ is independent of $\pr_{x'}$ under $\nu$ for all $x\in \BBB$. 
 \end{definition}
 A trivial spectral independence probability always exists, namely $\delta_{\emptyset_S}$ (strictly speaking we should restrict this Dirac measure --- defined by default on the whole of the power set of $S$ --- to $\Sigma$, but let us suffer this slight abuse of notation, as we do many others).

Directly from \eqref{spectral-sets:pi-system} and the definition,
 \begin{equation}\label{projections-composition}
\pr_y\circ\pr_x= \pr_{x\land y}=\pr_x\circ \pr_y\text{ a.e.-$\mu$ for }\{x,y\}\subset \BBB, \,  \pr_{0_\HH}=\emptyset_S\text{ a.e.-$\mu$} \text{ and }\pr_{1_\HH}=\mathrm{id}_S\text{ a.e.-$\mu$},
 \end{equation} 
which entails in particular that for a spectral independence probability $\nu$ the spectral projections $\pr_p$, $p\in P$, are independent under $\nu$ for all partitions of unity $P$ of $\BBB$. 

A straightforward  connection between the counting map $K$ and the spectral projections is 
 \begin{equation}\label{K-local}
 K=\sum_{x\in P}K(\pr_x)\text{ a.e.-$\mu$, indeed }K_b=\sum_{x\in P}K_b(\pr_x)\text{ a.e.-$\mu$ for all }b\in \mathfrak{F}\text{ for which }P\subset b,
 \end{equation}
this being so for all partitions of unity $P$ of $\BBB$, which we see from \eqref{eq:Kb} and \eqref{xK-limit}. 
By \eqref{projections-composition} the first equality of \eqref{K-local} is still true if $K$ is replaced with $K(\pr_y)$ for any $y\in \BBB$; if further $z\in \BBB_y$, we get $K(\pr_y)\geq K(\pr_y\circ\pr_z)= K(\pr_z)$ a.e.-$\mu$. Thus
\begin{equation}\label{K:nondecreasing}
\text{$K(\pr_u)\leq K(\pr_v)$ a.e.-$\mu$ for $u\subset v$ from $\BBB$};
\end{equation}
likewise, for $b\in \mathfrak{F}$ and $x\in b$:  $K_b(\pr_x)$ is $\mu$-a.e. $\uparrow$ in $x$, and $\mu$-a.e. $x\land \underline{b}$ is the join of precisely $K_b(\pr_x)$ atoms of $b$.
 
 Another property that we shall find useful is that 
 \begin{equation}\label{projections-and-subspaces}
  \HHH(\cap_{p\in P}\pr_p^{-1}(E_p))=\otimes_{p\in P}\HHH(E_p\cap S_p)=\otimes_{p\in P}[\HHH(E_p)\cap \HH_p]
 \end{equation}
 for all $E\in \Sigma^ P$ for all partitions of unity $P$ of $\BBB$. Indeed \eqref{projections-and-subspaces} is evidently true for $E_p$ of the form $S_{x_p}$ with $x_p\in \BBB_p$ as $p$ ranges over $P$, and extends easily to the general case by an application of Dynkin's lemma.

 \subsection{Spectrum of Fock unital factorizations}\label{subsection:spectrum-of-Fock}
  Consider the unital factorization $(\mathrm{Fock}(\int^\oplus\FF_s\nu(\dd s)),1)$ of Definition~\ref{definition:noise-factorization} situated on the Hilbert space $\boldsymbol{\Upgamma}_s(\int^\oplus\FF_s\nu(\dd s))$. Let $\hat\nu:=\sum_{n\in \mathbb{N}_0}(n!)^{-1}(q_n)_\star  \nu^n\vert_{(T^n)_\ne}$ be the symmetric measure over $\nu$, defined on the measurable space $((2^T)_{\mathrm{fin}},\TT_s)$, $q_n:(T^n)_{\ne}\to {T\choose n}$ being, for $n\in \mathbb{N}_0$, the map that sends an $n$-tuple with pairwise distinct entries to  its range, and the $\sigma$-field $\TT_s$ on $(2^T)_{\mathrm{fin}}$ being the largest one relative to which $q_n$ is $\Sigma^{\otimes n}\vert_{(T^n)_{\ne}}$-measurable for all $n\in \mathbb{N}_0$. We complete $((2^T)_{\mathrm{fin}},\TT_s,\hat \nu)$ to get the  standard measure space $((2^T)_{\mathrm{fin}},\tilde\TT,\tilde\nu)$. Further, we define a $\tilde\nu$-measurable field of separable Hilbert spaces $(\tilde\FF_A)_{A\in (2^T)_{\mathrm{fin}}}$ by setting 
  \begin{equation}\label{equation:tensor-product}
  \tilde\FF_A:=\otimes_{s\in A}\FF_s,\quad A\in (2^T)_{\mathrm{fin}},
  \end{equation}
  endowing it with the natural ``Fock'' $\tilde\nu$-measurable field structure characterized \cite[Propositions~II.1.2 and~II.1.4]{dixmier1981neumann} by the fact that the exponential vectors $E(u)$, $u\in\int^{\oplus}\FF_s\nu(\dd s)$, are total in $\int^\oplus \tilde\FF_A\tilde\nu(\dd A)$; here 
  \begin{equation*}
  E(u)_A:=\otimes_{a\in A}u(a),\quad A\in (2^T)_{\mathrm{fin}}.
  \end{equation*}
  
An elementary computation reveals that
  \begin{equation*}
  \langle E(u),E(v)\rangle=\int \prod_{a\in A}\langle u(a),v(a)\rangle \tilde\nu(\dd A)=e^{\int \langle u(s),v(s)\rangle\nu_c(\dd s)}\prod_{a\in\at(\nu)}\left(1+\nu(\{a\})\langle u(a),v(a)\rangle\right),\quad \{u,v\}\subset \int^{\oplus}\FF_s\nu(\dd s).
  \end{equation*}
  Hence, for $W\in \TT$, there exists a unique unitary isomorphism between $\left(\int^\oplus \tilde\FF_A\widetilde{\nu\vert_W}(\dd A)\right)\otimes\left(\int^\oplus \tilde\FF_A\widetilde{\nu\vert_{T\backslash W}}(\dd A)\right)$ and $\int^\oplus \tilde\FF_A\tilde\nu(\dd A)$, which carries $E(u)\otimes E(v)$ to $E(u\cup v)$ for all $u\in  \int^{\oplus}\FF_s[\nu\vert_W](\dd s)$ and all $v\in  \int^{\oplus}\FF_s[\nu\vert_{T\backslash W}](\dd s)$, the association being scalar product-preserving between total sets. By the very same token and \eqref{equation:scalar-product-in-exponential-Fock-space} there exists a unique unitary isomorphism $\tilde\Psi$ between $\boldsymbol{\Upgamma}_s(\int^\oplus\FF_s\nu(\dd s))$ and $\int^\oplus \tilde\FF_A\tilde\nu(\dd A)$, which sends $e_{\int^\oplus\FF_s\nu(\dd s)}(u)$ to $E(u)$ for all $u\in \int^\oplus\FF_s\nu(\dd s)$. Since the system $\{(2^W)_{\mathrm{fin}}: W\in \TT\}$ is generating for $\tilde\TT$ up to $\tilde\nu$-trivial sets \cite[Lemma~8.6(ii)]{vidmar-noise} we deduce that the pair $(\tilde{\nu}, \tilde\Psi)$ constitutes a spectral resolution for $(\mathrm{Fock}(\int^\oplus\FF_s\nu(\dd s)),1)$ with spectral sets, spectral projections   and counting map  given by
  \begin{equation*}
 \tilde S_{F_{\int ^\oplus\FF_s\nu(\dd s)}(W)}=(2^W)_{\mathrm{fin}}\text{ a.e.-$\tilde\nu$ for } W\in \TT,
  \end{equation*}
  \begin{equation*}
  \tilde\pr_{F_{\int ^\oplus\FF_s\nu(\dd s)}(W)}(A)=A\cap W\text{ for $\tilde\nu$-a.e. $A$ for all } W\in \TT,
  \end{equation*}
  and
    \begin{equation*}
\tilde K(A)=\vert A\vert \text{ for $\tilde\nu$-a.e. $A$},
  \end{equation*}
  respectively (using a $\tilde{}$ to indicate that the objects pertain to $(\tilde\nu,\tilde\Psi)$).   
  
  Precisely the same spectral resolution $(\tilde\nu,\tilde\Psi)$ as rendered above works for any unital subfactorization $(\FFF,1)$ of $(\mathrm{Fock}(\int^\oplus\FF_s\nu(\dd s)),1)$ provided  the algebra $\TT_\FFF:=\{W\in\TT:F_{\int ^\oplus\FF_s\nu(\dd s)}(W)\in \FFF\}$ generates $\TT$  (just because the system $\{(2^W)_{\mathrm{fin}}: W\in \TT_\FFF\}$ then still generates the domain of $\tilde\nu$ up to $\tilde\nu$-trivial sets \cite[Lemma~8.6(iii)]{vidmar-noise}). 
  
  We may say that the local tensor product form of the spectral resolution, being there always \eqref{eq:spectral-resolve-spaces}, becomes in the case of  a Fock unital factorization ``global''  \eqref{equation:tensor-product}.

\section{Characterizations of Fock form for unital factorizations}\label{section:fock-structure-completneess}
As a preliminary observation we note that thanks to separability of $\HH$, completeness of $(\BBB,\Omega)$ is just ``sequential completeness''.
\begin{proposition}
$\BBB$ is complete iff $\lor_{n\in \mathbb{N}}x_n\in\BBB$ for every $\uparrow$ sequence $(x_n)_{n\in \mathbb{N }}$ in $\BBB$.
\end{proposition}
\begin{proof}
The condition is trivially necessary. Sufficiency. $\BBB$ satisfies the ``countable chain condition'': no uncountable subset $X$ of $\BBB\backslash \{0_\HH\}$ has $x\land y=0_\HH$ for all $x\ne y$ from $X$; for, if such a subset did exist, then by Proposition~\ref{proposition:orthogonal} the $\HH_x\ominus (\mathbb{C}\Omega)$, $x\in X$, would be pairwise orthogonal and non-zero subspaces of $\HH$, which is in contradiction with the separability of $\HH$. By a well-known result \cite[Lemma~14.1]{halmos1963lectures} it follows that every subset $X$ of $\BBB$ admits a countable subset $Y$ which has the same  upper bounds (in the Boolean algebra $\BBB$);  but $\lor Y\in \BBB$ by assumption, so $\lor Y\supset \lor X$, whence (the opposite inclusion being obvious) $\lor X=\lor Y\in \BBB$. 
\end{proof}
%

In this section we prove 
\begin{theorem}\label{thm:main-for-noise-factorization}
The following are equivalent.
\begin{enumerate}[(A)]
\item\label{classical:A} There is a complete unital factorization of $\HH$ which has $(\BBB,\Omega)$ as a subfactorization.
\item\label{classical:B}  For every sequence $(x_n)_{n\in \mathbb{N}}$ in $\BBB$, $\lor_{n\in \mathbb{N}}x_n$ is a factor, i.e. $(\lor_{n\in \mathbb{N}}x_n)\lor (\land_{n\in \mathbb{N}}x_n')=1_\HH$.
\item\label{classical:C} $K<\infty$ a.e.-$\mu$.
\item\label{classical:D} $(\BBB,\Omega)$ is of Fock type.
\item\label{classical:E} The multiplicative vectors are total.
\item\label{classical:F} There exists a spectral independence probability equivalent to $\mu$.
\end{enumerate}
\end{theorem}
We do this by establishing, through a series of results occupying Subsections~\ref{subsection:finiteness}-\ref{subsection:multiplicative-carried}, the implications \ref{classical:A} $\Rightarrow$ \ref{classical:B} (Proposition~\ref{AB}),  \ref{classical:B} $\Rightarrow$ \ref{classical:C} (Proposition~\ref{BC}),   \ref{classical:C} $\Rightarrow$ \ref{classical:E}  (Corollary~\ref{CE}), \ref{classical:C} $\Rightarrow$ \ref{classical:D}  (Corollary~\ref{CD}), \ref{classical:D} $\Rightarrow$ \ref{classical:A} (Proposition~\ref{DA}), \ref{classical:C} $\Rightarrow$ \ref{classical:F} (Proposition~\ref{CF}) and [\ref{classical:E} or \ref{classical:F}] $\Rightarrow$ \ref{classical:C} (Corollary~\ref{EF,C}). 

\begin{definition}\label{definition:classical}
We call $(\BBB,\Omega)$ classical if it meets one and then all of the conditions of Theorem~\ref{thm:main-for-noise-factorization}.
\end{definition}
The notion of blackness, which is a kind of opposite of classicality, is considered briefly in Subsection~\ref{subsection:blackness}. In the same subsection, the reader will find  Extended remark~\ref{extendedremark:noise-Boolean}, which especially one well-versed in the classical probabilistic counterpart of this theory should find a great aid in the understaning of Theorem~\ref{thm:main-for-noise-factorization}.  Indeed, as a guide to intuition, Extended remark~\ref{extendedremark:noise-Boolean}  might even ``preemptively'' be consulted  already at this point, before attempting the proofs.

 \subsection{Finiteness of the counting map}\label{subsection:finiteness}
  Elementary is
 \begin{proposition}\label{AB}
\ref{classical:A} implies \ref{classical:B}.
 \end{proposition}
 \begin{proof}
 $\lor_{n\in \mathbb{N}}x_n$ is  a factor since it belongs to the given complete factorization, which contains $\BBB$. It remains to note that $(\lor_{n\in \mathbb{N}}x_n)'=\land_{n\in \mathbb{N}}x_n'$.
 \end{proof}
 We examine the implications of  \ref{classical:B} on $K$.
\begin{lemma}\label{lemma:for-K-finite}
If \ref{classical:B} holds true, then for every sequence $(x_n)_{n\in \mathbb{N}}$ in $\BBB$, $\lor_{m\in \mathbb{N}}\land_{n\in \mathbb{N}}(x_1\lor\cdots\lor x_m)\lor (x_1\lor\cdots\lor x_n)'=1_\HH$, hence  
\begin{equation*}
\text{ for $\mu$-a.e. $s$: } \exists m\in \mathbb{N}\backepsilon \forall n\in \mathbb{N}:\, s\in S_{(x_1\lor\cdots\lor x_m)\lor (x_1\lor\cdots\lor x_n)'}.
\end{equation*}
\end{lemma}
\begin{proof}
Put $\overline{x}:=\lor_{n\in \mathbb{N}}x_n$. Then $(x_1\lor\cdots\lor x_n)'\supset \overline{x}'$ for all $n\in \mathbb{N}$. Thus $\lor_{m\in \mathbb{N}}\land_{n\in \mathbb{N}}(x_1\lor\cdots\lor x_m)\lor (x_1\lor\cdots\lor x_n)'\supset \lor_{m\in \mathbb{N}}[(x_1\lor\cdots\lor x_m)\lor \overline{x}']=\overline{x}\lor \overline{x}'=1_\HH$, which means that the inclusion is actually an equality. On the other hand, from $\lor_{m\in \mathbb{N}}\land_{n\in \mathbb{N}}(x_1\lor\cdots\lor x_m)\lor (x_1\lor\cdots\lor x_n)'=1_\HH$ and $\phi_{\lor_{m\in \mathbb{N}}\land_{n\in \mathbb{N}}(x_1\lor\cdots\lor x_m)\lor (x_1\lor\cdots\lor x_n)'}=\lor_{m\in \mathbb{N}}\phi_{\land_{n\in \mathbb{N}}(x_1\lor\cdots\lor x_m)\lor (x_1\lor\cdots\lor x_n)'}\leq \lor_{m\in \mathbb{N}}\land_{n\in \mathbb{N}}\phi_{(x_1\lor\cdots\lor x_m)\lor (x_1\lor\cdots\lor x_n)'}$ we get  $\mathbf{1}_\HH\geq \lor_{m\in \mathbb{N}}\land_{n\in \mathbb{N}}\phi_{(x_1\lor\cdots\lor x_m)\lor (x_1\lor\cdots\lor x_n)'}\geq \phi_{1_\HH}=\mathbf{1}_\HH$ and hence $S= \cup_{m\in \mathbb{N}}\cap_{n\in \mathbb{N}}S_{(x_1\lor\cdots\lor x_m)\lor (x_1\lor\cdots\lor x_n)'}$ a.e.-$\mu$, since the $^*$-isomorphism $\alpha$ respects bounded $\uparrow$ suprema and $\downarrow$ infima.
\end{proof}
With Lemma~\ref{lemma:for-K-finite} in hand the proof of the proposition to follow is now essentially verbatim the same as in the commutative case \cite[pp. 350-352]{tsirelson}. We repeat it here for the reader's convenience. It is more succinct relative to the original because the spectral projections (which were only introduced later in \cite{vidmar-noise}) allow for a greater degree of economy of expression. 
\begin{proposition}\label{BC}
\ref{classical:B} implies \ref{classical:C}.
\end{proposition}
\begin{proof}
For the moment we do not assume \ref{classical:B}. Let $(b_n)_{n\in \mathbb{N}}$ be an exhausting sequence in $\mathfrak{F}$ (recall \eqref{xK-limit}).


Next we choose sequences $(p_n)_{n\in \mathbb{N}}$ in $(0,1)$ and $(c_n)_{n\in \mathbb{N}}$ in $\mathbb{N}$ such that 
\begin{equation}\label{conditions:c-p}
\sum_{n\in \mathbb{N}}p_n<1
\text{ and } \lim_{n\to\infty}(1-p_n)^{c_n}=0
\end{equation}
 (they exist, any will do).  Passing to an equivalent finite measure we may and do assume $\mu$ is finite. Due to \eqref{xK-limit}, by continuity of $\mu$ from above, there is a sequence $(\mathsf{n}_n)_{n\in \mathbb{N}}$ in $\mathbb{N}$ that is $\uparrow\uparrow\infty$ and such that 
\begin{equation*}
\sum_{n\in \mathbb{N}}\mu(K_{b_{\mathsf{n}_n}}<c_n,K=\infty)<\infty; 
\end{equation*}
by Borel-Cantelli we deduce that 
\begin{equation}\label{eq:funny-condition}
\text{[either $K<\infty$ or else ($K_{b_{\mathsf{n}_n}}\geq c_n$ for all $n\in \mathbb{N}$ large enough)] a.e.-$\mu$}.
\end{equation}
For ease of notation, passing to a subsequence of $(b_n)_{n\in \mathbb{N}}$ if necessary, we may and do assume (vis-\`a-vis the validity of \eqref{conditions:c-p}-\eqref{eq:funny-condition})  that 
\begin{equation}\label{eq:wlog-for-n}
\text{$\mathsf{n}_n=n$ for all $n\in \mathbb{N}$.}
\end{equation} 

Introduce now the product probability $\PP:=\times_{n\in \mathbb{N}}\PP_n$ on the  measurable space $(\prod_{n\in \mathbb{N}}b_n,\otimes_{n\in \mathbb{N}}2^{b_n})$, where, for each $n\in \mathbb{N}$, $\PP_n$ is the unique probability on $(b_n,2^{b_n})$ such that 
\begin{equation*}
\PP_n(\{x\})=p_n^{\vert \at(b_n)\cap 2^x\vert}(1-p_n)^{\vert \at(b_n)\vert-\vert \at(b_n)\cap 2^x\vert},\quad x\in b_n;
\end{equation*} in words, $\PP_n$ includes each atom of $b_n$ with probability $p_n$ independently of the others.  We denote by $X=(X_n)_{n\in \mathbb{N}}$ the coordinate process on $\prod_{n\in \mathbb{N}}b_n$. The probability $\PP$ may then equivalently be described in the following manner: the $X_n$, $n\in \mathbb{N}$, are independent under $\PP$ and ${X_n}_\star\PP=\PP_n$ for all $n\in \mathbb{N}$. Finally we define the ``running joins''
\begin{equation*}
Y_n:=X_1\lor \cdots \lor X_n,\quad n\in \mathbb{N},
\end{equation*}
 and record two ancillary claims (which are generally valid).
 
 \begin{lemma}\label{lemma:one}
For $\mu$-a.e. $s\in \{K=\infty\}$, $$\PP(K(\pr_{Y_m'}(s))<\infty\text{ for some $m\in \mathbb{N}$})< 1.$$
\end{lemma}
\begin{proof}
 For $\mu$-a.e. $s\in \{K=\infty\}$ we have as follows. For all $m\in \mathbb{N}$, by \eqref{K-local} applied with $P=\at(b_m)$, there is $a\in \at(b_m)$ such that $K(\pr_a(s))=\infty$. By \eqref{projections-composition}-\eqref{K-local}, $K(\pr_{Y'_m}(s))=\sum_{x\in \at(b_m)\cap 2^{Y'_m}}K(\pr_x(s))$ and we get
\begin{equation*}
\{K(\pr_{Y_m'}(s))<\infty\}\subset  \{a\nsubset Y_m'\} =\{a\subset Y_m\}\subset \cup_{n\in [m]}\{a\subset X_n\}.
\end{equation*}
Applying monotonicity and finite subaddivity of $\PP$ we get $\PP(K(\pr_{Y_m'}(s))<\infty)\leq \sum_{n=1}^mp_n$. Passing to the limit $m\to\infty$ using \eqref{K:nondecreasing} and the continuity of $\PP$ from below the claim now follows from the first condition of  \eqref{conditions:c-p}.
\end{proof}


\begin{lemma}\label{lemma:three}
For $\mu$-a.e. $s$,
\begin{equation*}
\PP\left(\exists m\in \mathbb{N}\left[ K(\pr_{Y_m'}(s))=\infty\text{ and } s\in S_{Y_m\lor Y_n'}\text{ for all $n\in \mathbb{N}$}\right]\right)=0.
\end{equation*}
\end{lemma}
\begin{proof}
The argument is again for $\mu$-a.e. $s$. Fix $m\in \mathbb{N}$. Recall from \eqref{eq:ad-def-K} that for all $n\in \mathbb{N}_{\geq m}$, $s\in S_{Y_m\lor Y_n'}$ iff $Y_m\lor Y_n'\supset \underline{b_n}(s)$. Separating the event $\{K(\pr_{Y_m'}(s))=\infty\text{ and }s\in S_{Y_m\lor Y_n'}\text{ for all }n\in \mathbb{N}_{\geq m}\}$ according to the possible (finitely many) values that $Y_m'$ can take in $b_m$ it suffices to prove (by subadditivity of $\PP$) that for all $x\in b_m$, if $s\in \{K(\pr_x)=\infty\}$, then 
\begin{equation*}
\PP(x'\lor Y_n'\supset \underline{b_n}(s)\text{ for all }n\in \mathbb{N}_{\geq m})=0.
\end{equation*}
But for $n\in \mathbb{N}_{\geq m}$, $\{x'\lor Y_n'\supset \underline{b_n}(s)\}=\{(Y_n\land x)'\supset \underline{b_n}(s)\}=\{Y_n\land x\land \underline{b_n}(s)=0_\HH\}\subset\{X_n\land x\land \underline{b_n}(s)=0_\HH\}=\{X_n\subset (x\land \underline{b_n}(s))'\}$. Furthermore, 
$x\land \underline{b_n}(s)$ is the join of precisely $K_ {b_n}(\pr_x(s))$ atoms of $b_n$. Therefore $\PP(X_n\subset (x\land \underline{b_n}(s))')=(1-p_n)^{K_{b_n}(\pr_x(s))}$, which, if $s\in \{K(\pr_x)=\infty\}$, is $\to 0$ as $n\to\infty$ by \eqref{eq:funny-condition}-\eqref{eq:wlog-for-n} and the second condition of \eqref{conditions:c-p}.
 \end{proof}

Now, assuming \ref{classical:B}, Lemmas~\ref{lemma:for-K-finite} and~\ref{lemma:three} tell us that 
\begin{equation*}
\PP(K(\pr_{Y_m'}(s))<\infty\text{ for some $m\in \mathbb{N}$})=1
\end{equation*}
for $\mu$-a.e. $s$. Combined with Lemma~\ref{lemma:one} it renders $\mu(K=\infty)=0$, which is \ref{classical:C}.
\end{proof}

\subsection{Fock form} Let us investigate  the ``classical part'' $\HHH(\{K<\infty\})$. Its structure will be revealed fully in Proposition~\ref{proposition:fock-spectrum} via the exponential map of Proposition~\ref{prop:Exp}. But first we exhaust it through suitable tensor products in
\begin{proposition}\label{proposition:products-total}
The collection
\begin{equation*}
\left\{\otimes_{p\in P}X_p:X\in \prod_{p\in P}\HHH(\{K\leq 1\})\cap \HH_p\text{, $P$ a partition of unity of $\BBB$}\right\}
\end{equation*}
is total in $\HHH(\{K<\infty\})$.
\end{proposition}
\begin{proof}
The family is included in $\HHH(\{K<\infty\})$ due to \eqref{K-local}-\eqref{projections-and-subspaces}. Let $(b_n)_{n\in \mathbb{N}}$ be an exhausting sequence in $\mathfrak{F}$. By  \eqref{projections-composition}-\eqref{K-local}, for $n\in \mathbb{N}$ and $m\in \mathbb{N}_0$, $\{K=K_{b_n}=m\}=\cup_{Q\in { \at(b_n)\choose m}}\left[\left(\cap_{a\in Q}\pr_a^{-1}(\{K=1\})\right)\cap\left(\cap_{a\in \at(b_n)\backslash Q}\pr_a^{-1}(\{K=0\})\right)\right]$ a.e.-$\mu$. The evident a.e.-$\mu$ equality $\{K<\infty\}=\cup_{m\in \mathbb{N}_0} \cup_{n\in \mathbb{N}}\{K=K_{b_n}=m\}$ thus becomes 
\begin{equation*}
\{K<\infty\}=\cup_{n\in \mathbb{N}}\cup_{m\in \mathbb{N}_0}\cup_{Q\in { \at(b_n)\choose m}}\left[\left(\cap_{a\in Q}\pr_a^{-1}(\{K=1\})\right)\cap\left(\cap_{a\in \at(b_n)\backslash Q}\pr_a^{-1}(\{K=0\})\right)\right] \text{ a.e.-$\mu$,}
\end{equation*} which in view of  \eqref{projections-and-subspaces}  yields the stipulated totality.
\end{proof}

The following technical measure-theoretic result shall be employed below.
\begin{lemma}\label{lemma:limit-involving-complex-measure}
Let $\nu$ be a complex measure on a countably separated measurable space $(X,M)$. Let also $\mathcal{P}=(\mathcal{P}_n)_{n\in \mathbb{N}}$ be a sequence of finite $M$-measurable partitions of $X$, whose union is separating the points of $X$ and such that $\mathcal{P}_{n+1}$ is finer than $\mathcal{P}_n$ for all $n\in \mathbb{N}$ (such sequences exist precisely because $(X,M)$ is countably separated). In other words, $\mathcal{P}$ is a dissecting system in the sense of  \cite[Definition~A1.6.I]{vere-jones}. Then 
\begin{equation}\label{eq:a.measure-limit}
\lim_{n\to\infty}\prod_{P\in \mathcal{P}_n}(1+\nu(P))=e^{\nu(X\backslash \at(\vert\nu\vert))}\prod_{x\in \at(\vert\nu\vert)}(1+\nu(\{x\})).
\end{equation}
\end{lemma}
\begin{proof}
For $n\in \mathbb{N}_0$ and $x\in \mathbb{C}^{[n]}$ we have 
\begin{equation}\label{elementary}
\left\vert \prod_{i\in [n]}(1+x_i)-1\right\vert\leq \prod_{i\in [n]}(1+\vert x_i\vert)-1\leq e^{\sum_{i\in [n]}\vert x_i\vert}-1\leq \left(\sum_{i\in [n]}\vert x_i\vert\right) e^{\sum_{i\in [n]}\vert x_i\vert}.
\end{equation} 
Notice also that for all $x\in  X$, if we denote, for $n\in \mathbb{N}$, by $P_{n,x}$ the unique element of $\mathcal{P}_n$ that contains $x$, then 
\begin{equation}\label{descending-to-singleton}
\text{$P_{n,x}\downarrow \{x\}$ as $n\to\infty$}
\end{equation}
(because $\cup_{n\in \mathbb{N}}\mathcal{P}_n$ is separating $X$).

For any non-empty $P\subset X$ denote next by $a_P$ an arbitrary element $p\in P$ which maximizes $\vert\nu\vert(\{p\})$ (so, $\nu(\{a_P\})=\max\{\vert \nu\vert(\{p\}):p\in P\}$, the maximum being attained because $\vert \nu\vert$ is finite). We claim that 
\begin{equation}\label{eq:uniform-conv}
\text{$\downarrow$-$\lim_{n\to\infty}$}\max_{P\in \mathcal{P}_n}\vert \nu\vert(P\backslash \{a_P\})=0.
\end{equation}
The argument  is a straightforward extension of \cite[proof of Lemma~A1.6.II]{vere-jones}, which handles the nonatomic case. As it is very short we may as well include it here. Indeed, if \eqref{eq:uniform-conv} were not the case, then for some $\epsilon>0$ we would find a sequence $(P_n)_{n\in \mathbb{N}}$ with $P_n\in \mathcal{P}_n$ and $\vert \nu\vert(P_n\backslash \{a_{P_n}\})\geq \epsilon$ for all $n\in \mathbb{N}$  (call this property $(\dagger)$).  By Zorn (or just by direct construction) there is  a subset  $A$ of $\mathbb{N}$ maximal w.r.t. inclusion having the property that the $P_n$, $n\in A$, are pairwise disjoint.  If $A$ is infinite it contradicts finiteness of $\vert\nu\vert$. Otherwise there is $a\in A$ such that $P_n\subset P_a$ for infinitely many $n\in \mathbb{N}_{>a}$. Proceeding inductively we deduce existence of a nonincreasing (w.r.t. inclusion) subsequence $(P_{\mathsf{n}_k})_{k\in \mathbb{N}}$ of $(P_n)_{n\in \mathbb{N}}$. Now, 
by the separating property, $\cap_{k\in \mathbb{N}}P_{\mathsf{n}_k}$ contains at most one point $x$, and if this point is further an atom of $\vert\nu\vert$, there being only finitely many other atoms of size no smaller than $\vert\nu\vert(\{x\})$, $a_{P_{\mathsf{n}_k}}$ must be equal to $x$ for all but finitely many $k\in \mathbb{N}$ (again by the separating property). Hence, whether or not $\cap_{k\in \mathbb{N}}P_{\mathsf{n}_k}$ is empty, $\lim_{k\to\infty}\vert\nu\vert(P_{\mathsf{n}_k}\backslash \{a_{P_{\mathsf{n}_k}}\})=0$ contradicting $(\dagger)$. 

By bounded convergence \eqref{eq:uniform-conv} entails 
\begin{equation}\label{eq:weird-with-indicators}
\lim_{n\to\infty}\vert \nu\vert\left[\sum_{P\in \mathcal{P}_n}\vert \nu\vert(P\backslash \{a_P\})\mathbbm{1}_{P}\right]=0.
\end{equation}

Write now $Q_n:=\{a_P:P\in \mathcal{P}_n\}$ for $n\in \mathbb{N}$. Using the elementary estimate \eqref{elementary} we reduce checking \eqref{eq:a.measure-limit} to the case when $\nu(\{x\})\ne -1$ for all $x\in \at(\vert\nu\vert)$, note that as a consequence of this and of the finiteness of $\vert \nu\vert$,
\begin{equation}\label{eq:uniform-atoms}
\inf_{x\in \at(\vert\nu\vert)}\vert 1+\nu(\{x\})\vert>0,
\end{equation}
and  resolve to evaluating
\begin{align*}
\lim_{n\to\infty}\prod_{P\in \mathcal{P}_n}(1+\nu(P))&=\lim_{n\to\infty}\left(\prod_{q\in Q_n}(1+\nu(\{q\}))\right)\prod_{P\in \mathcal{P}_n}\frac{1+\nu(P)}{1+\nu(\{a_P\})}\\
&=\left(\prod_{x\in \at(\vert\nu\vert)}(1+\nu(\{x\}))\right)\lim_{n\to\infty}\prod_{P\in \mathcal{P}_n}\frac{1+\nu(P)}{1+\nu(\{a_P\})}\quad (\because \text{ \eqref{elementary}})\\
&=\left(\prod_{x\in \at(\vert\nu\vert)}(1+\nu(\{x\}))\right)\lim_{n\to\infty}\prod_{P\in \mathcal{P}_n}(1+\nu(P\backslash \{a_P\})),
\end{align*}
taking into account in the last equality that,  again in virtue of \eqref{elementary},
\begin{equation*}
\lim_{n\to\infty}\prod_{P\in \mathcal{P}_n}\frac{1+\nu(P)}{(1+\nu(\{a_P\}))(1+\nu(P\backslash \{a_P\}))}=1,
\end{equation*}
since
\begin{align*}
\sum_{P\in  \mathcal{P}_n}\left\vert\frac{1+\nu(P)}{(1+\nu(\{a_P\}))(1+\nu(P\backslash \{a_P\}))}-1\right\vert&\leq \sum_{P\in  \mathcal{P}_n}\frac{\vert \nu\vert(\{a_P\})\vert\nu\vert(P\backslash \{a_P\})}{\vert( 1+\nu(\{a_P\}))(1+\nu(P\backslash \{a_P\}))\vert}\\
&\leq \frac{\vert \nu\vert(X)\max_{P\in \mathcal{P}_n}\vert \nu\vert(P\backslash \{a_P\})}{\left[1\land \left(\inf_{x\in \at(\vert\nu\vert)}\vert 1+\nu(\{x\})\vert\right)\right]\left(1-\max_{P\in \mathcal{P}_n}\vert \nu\vert(P\backslash \{a_P\})\right)}\to 0
\end{align*}
as $n\to\infty$ by \eqref{eq:uniform-conv} and \eqref{eq:uniform-atoms}.

Yet another estimate using \eqref{elementary} shows that in checking \eqref{eq:a.measure-limit} at this point we may and do assume $\nu$ is atomless. Then, recalling \eqref{eq:uniform-conv}, we express, taking the principal branch of the complex logarithm,
\begin{equation*}
\lim_{n\to\infty}e^{-\nu(X)}\prod_{P\in \mathcal{P}_n}(1+\nu(P))=\lim_{n\to\infty}e^{\sum_{P\in \mathcal{P}_n}(\log(1+\nu(P))-\nu(P))};
\end{equation*}
but,  at least for all large enough $n\in \mathbb{N}$,
\begin{equation*}
\left\vert \sum_{P\in \mathcal{P}_n}(\log(1+\nu(P))-\nu(P))  \right\vert\leq \sum_{P\in \mathcal{P}_n}\frac{(\vert\nu\vert (P))^2}{2(1-\vert\nu\vert (P))}\to 0 ,
\end{equation*}
in virtue of \eqref{eq:uniform-conv} and since $\sum_{P\in \mathcal{P}_n}(\vert\nu\vert (P))^2=\vert \nu\vert[\sum_{P\in \mathcal{P}_n}\vert\nu\vert (P)\mathbbm{1}_P]\to 0$ as $n\to \infty$ by \eqref{eq:weird-with-indicators}.
\end{proof}
Also, an elementary inequality. 

\begin{lemma}\label{lemma:elementary-inequality}
For $n\in \mathbb{N}$ and $x\in [0,\infty)^{[n]}$, 
\begin{equation*}
\sum_{\substack{A\in 2^{[n]}, \\\vert A\vert\geq 2}} \prod_{i\in A}x_i\leq  \land_{j\in [n]}\left(\sum_{i\in [n]\backslash \{j\}}x_i\right)\left(\sum_{i\in [n]}x_i\right)e^{\sum_{i\in [n]}x_i}.
\end{equation*}
\end{lemma}
\begin{proof}
We compute and estimate
 \begin{align*}
 \sum_{\substack{A\in 2^{[n]}, \\\vert A\vert\geq 2}} \prod_{i\in A}x_i&=\prod_{i\in [n]}(1+x_i)-1-\sum_{i\in [n]}x_i\\
 &= (1+x_n)\prod_{i\in [n-1]}(1+x_i)-1-x_n-\sum_{i\in [n-1]}x_i\\
&\leq (1+x_n)e^{\sum_{i\in [n-1]}x_i}-(1+x_n)-\sum_{i\in [n-1]}x_i=(1+x_n)\left(e^{\sum_{i\in [n-1]}x_i}-1\right)-\sum_{i\in [n-1]}x_i\\
&\leq e^{x_n}\left(\sum_{i\in [n-1]}x_i\right)e^{\sum_{i\in [n-1]}x_i}-\sum_{i\in [n-1]}x_i=\left(\sum_{i\in [n-1]}x_i\right)\left(e^{\sum_{i\in [n]}x_i}-1\right)\\
&\leq \left(\sum_{i\in [n-1]}x_i\right)\left(\sum_{i\in [n]}x_i\right)e^{\sum_{i\in [n]}x_i}.
\end{align*}
Due to the symmetry of the left-hand side in the $x_i$, $i\in [n]$, the desired conclusion follows.
\end{proof}
With Lemmas~\ref{lemma:limit-involving-complex-measure} and~\ref{lemma:elementary-inequality} in hand we are ready to introduce the exponential map $\Exp:\HHH(\{K=1\})\to \HHH(\{K<\infty\})$, which mimicks  $e_{\int^{\oplus}\FF_s\nu(\dd s)}:\int^{\oplus}\FF_s\nu(\dd s)\to \boldsymbol{\Upgamma}_s(\int^{\oplus}\FF_s\nu(\dd s))$ of Definition~\ref{definition:noise-factorization}. The construction is a remake of the partial commutative result of \cite[Lemma~A5]{vershik-tsirelson} with some  twists to handle the full generality.
\begin{proposition}\label{prop:Exp}
Let $g\in \HHH(\{K=1\})$. We have the following assertions.
\begin{enumerate}[(i)]
\item\label{prop:Exp-1} $g=\oplus_{p\in P}\phi_p(g)$ for all partitions of unity $P$ of $\BBB$. Also, $\phi_{0_\HH}(g)=0$, i.e. $\langle \Omega,g\rangle=0$. And $\phi_x(g)\in \HHH(\{K=1\})\cap \HH_x$ for all $x\in\BBB$.
\item\label{prop:Exp-2} The limit 
\begin{equation}\label{eq:exponential}
\Exp(g):=\lim_{n\to\infty}\otimes_{p\in \at(b_n)}(\Omega+\phi_p(g))
\end{equation}
exists for any $\uparrow$ sequence $(b_n)_{n\in \mathbb{N}}$ in $\mathfrak{F}$ satisfying the property that $\cup_{n\in \mathbb{N}}\{S_{x}:x\in b_n\}$ generates $\Sigma$ on $ \{K=1\}$ up to $\mu$-trivial sets (equivalently, by standardness, separates the points of a $\mu$-conegligible set of $\{K=1\}$); 
such sequences exist and the limit does not depend on the choice of the sequence.
\item\label{prop:Exp-3} $\Exp(g)$ is a multiplicative vector. For $x\in \BBB$, $\phi_x(\Exp(g))=\Exp(\phi_x(g))$. Besides, $\Exp(0)=\Omega$.
\item\label{prop:Exp-4} 
 Let also $h\in \HHH(\{K=1\})$. Then
\begin{equation}\label{scalar-product-exps}
\langle \Exp(h),\Exp(g)\rangle=e^{\mu_{h,g}(S\backslash\at(\mu))}\prod_{s\in \at(\mu)}(1+\mu_{h,g}(\{s\})).
\end{equation}
 \item\label{prop:Exp-5}  The family
\begin{equation*}
\{\Exp(f):f\in \HHH(\{K=1\})\}
\end{equation*}
is total in $\HHH(\{K<\infty\})$.

\end{enumerate}
\end{proposition}
\begin{proof}
\ref{prop:Exp-1}.  By the very definition of $K$, a.e.-$\mu$ on $\{K=1\}$, $S=\sqcup_{p\in P}S_p$ (disjoint union), which gives the first claim. The second assertion  follows from the two facts that $\HHH(\{K=1\})$ is orthogonal to $\HHH(\{K=0\})=\HHH(\{\emptyset_S\})=\mathbb{C}\Omega$ and that $\phi_{0_\HH}=\vert\Omega\rangle\langle\Omega\vert$. The last observation is immediate from the fact that $\Psi(\phi_x(g))=\mathbbm{1}_{S_x}\Psi(g)$.

\ref{prop:Exp-2}. Take any $b\subset c$ from $\mathfrak{F}$. We compute and estimate
\begin{align}
\nonumber&\Vert \otimes_{p\in \at(b)}(\Omega+\phi_p(g))-\otimes_{q\in \at(c)}(\Omega+\phi_q(g))\Vert^2\\\nonumber
&=\Vert \otimes_{p\in \at(b)}(\Omega+\oplus_{q\in \at(c)\cap 2^p}\phi_q(g))-\otimes_{q\in \at(c)}(\Omega+\phi_q(g))\Vert^2 \\\nonumber
&=\Vert \oplus_{A\in 2^{\at(c)}, \exists p\in \at(b)\backepsilon \vert A\cap 2^p \vert\geq 2}\otimes_{q\in A}\phi_q(g)\Vert^2\\\nonumber
&=\sum_{\substack{A\in 2^{\at(c)},\\ \exists p\in \at(b)\backepsilon \vert A\cap 2^p \vert\geq 2}}\prod_{q\in A}\Vert\phi_q(g)\Vert^2\\\nonumber
&\leq \sum_{p\in \at(b)}\sum_{\substack{A\in 2^{\at(c)},\\ \vert A\cap 2^p\vert\geq 2}}\prod_{q\in A}\Vert\phi_q(g)\Vert^2\\\nonumber
&=\sum_{p\in \at(b)}\left(\sum_{\substack{A\in 2^{\at(c)\cap 2^p},\\\vert A\vert\geq 2}}\prod_{q\in A}\Vert\phi_q(g)\Vert^2\right)\prod_{r\in \at(b)\backslash \{p\}}\left(\sum_{A\in 2^{\at(c)\cap 2^r}}\prod_{q\in A}\Vert\phi_q(g)\Vert^2\right)\\\nonumber
&=\sum_{p\in \at(b)}\left(\sum_{\substack{A\in 2^{\at(c)\cap 2^p},\\ \vert A \vert\geq 2}}\prod_{q\in A}\Vert\phi_q(g)\Vert^2\right)\prod_{r\in \at(b)\backslash \{p\}}\prod_{q\in \at(c)\cap 2^r}(1+\Vert \phi_q(g)\Vert^2)\\\nonumber
&\leq e^{\Vert g\Vert^2}\sum_{p\in \at(b)}\left(\Vert\phi_p(g)\Vert^2-\max_{q\in \at(c)\cap 2^p}\Vert \phi_q(g)\Vert^2\right)\Vert\phi_p(g)\Vert^2\\
&= e^{\Vert g\Vert^2}\mu_g\left[\sum_{p\in \at(b)}\land_{q\in \at(c)\cap 2^p}\mu_g(S_p\backslash S_q)\mathbbm{1}_{S_p}\right],\label{eq:main-estimate}
\end{align}
where 
\begin{itemize}
 \item the first equality uses \ref{prop:Exp-1}, 
 \item the second one is got by  expanding the tensor products, canceling the terms got from $\otimes_{p\in \at(b)}(\Omega+\oplus_{q\in \at(c)\cap 2^p}\phi_q(g))$ with some of those of $\otimes_{p\in \at(c)}(\Omega+\phi_p(g))$, the remaining terms in evidence being orthogonal due to \ref{prop:Exp-1} again (recall also \eqref{identification-another'}),
 \item the second inequality uses  that $1+v\leq e^v$ for $v\in [0,\infty)$, as well as Lemma~\ref{lemma:elementary-inequality} together with $\Vert g\Vert^2=\sum_{q\in \at(c)}\Vert \phi_q(g)\Vert^2$ and $\Vert \phi_p(g)\Vert^2=\sum_{q\in \at(c)\cap 2^p}\Vert \phi_q(g)\Vert^2$ for $p\in \at(b)$,
\end{itemize}
while the remaining steps we feel are sufficiently self-explanatory as to not warrant further comment.

\eqref{eq:main-estimate} gives (the Cauchy property and hence) existence of the limit \eqref{eq:exponential} by bounded convergence, since for $\mu$-a.e. $s\in \{K=1\}$, $S_{\underline{b_n}(s)}\downarrow \{s\}$ as $n\to\infty$ a.e.-$\mu$ by the assumed property on $(b_n)_{n\in \mathbb{N}}$. 

If $(\tilde b_n)_{n\in\mathbb{N}}$ is another sequence such as $(b_n)_{n\in\mathbb{N}}$ then so too is $(b_n\lor \tilde b_n)_{n\in \mathbb{N}}$ and we easily conclude uniqueness of the limit using again \eqref{eq:main-estimate}. 

Lastly, a sequence $(b_n)_{n\in\mathbb{N}}$ with the stipulated property surely exists because a countable subfamily of $\{S_x:x\in \BBB\}$ generates $\Sigma$ up to $\mu$-trivial sets.

\ref{prop:Exp-3}. For any fixed $x\in \BBB$ we may insist in  \eqref{eq:exponential} that $x\in b_1$.

\ref{prop:Exp-4}. 
The same sequence $(b_n)_{n\in \mathbb{N}}$ may be, and is here taken for $g$ and $h$ in defining $\Exp(g)$ and $\Exp(h)$ through \eqref{eq:exponential} respectively. 
Then we compute
\begin{align*}
\langle \Exp(h),\Exp(g)\rangle&=\lim_{n\to\infty}\prod_{p\in \at(b_n)}\left(1+\langle \phi_p(h),\phi_p(g)\rangle\right)=\lim_{n\to\infty}\prod_{p\in \at(b_n)}\left(1+\mu_{h,g}(S_p)\right)\\
&=e^{\mu_{h,g}(S\backslash\at(\mu))}\prod_{s\in \at(\mu)}(1+\mu_{h,g}(\{s\})),
\end{align*}
where we have applied Lemma~\ref{lemma:limit-involving-complex-measure} (after discarding some $\mu$-negligible set) in the last equality.

\ref{prop:Exp-5}. Each term of the limit \eqref{eq:exponential} belongs to the collection of Proposition~\ref{proposition:products-total}, a fortiori to $\HHH(\{K<\infty\})$. 
Next  notice that 
\begin{equation*}
\lim_{\epsilon\downarrow 0}\Exp(\epsilon g)=\Omega,
\end{equation*}
since, in the setting of \eqref{eq:exponential},
\begin{equation*}
\Vert \Exp(\epsilon g)- \Omega\Vert^2=\lim_{n\to\infty}\sum_{A\in 2^{\at(b_n)}\backslash \{\emptyset\}}\prod_{a\in A}\epsilon^{2}\Vert \phi_a(g)\Vert^2= \lim_{n\to\infty}\prod_{a\in \at(b_n)}(1+\epsilon^2\Vert \phi_a(g)\Vert^2)-1 \leq e^{\epsilon^2\Vert g\Vert^2}-1,
\end{equation*}
which is $\to 0$ as $\epsilon\downarrow 0$; moreover, 
\begin{equation*}
\lim_{\epsilon\downarrow 0}\frac{\Exp(\epsilon g)-\Omega}{\epsilon}=g,
\end{equation*}
because, again in the setting of \eqref{eq:exponential}, 
\begin{align*}
\left\Vert \frac{\Exp(\epsilon g)- \Omega}{\epsilon}-g\right\Vert^2&=\epsilon^{-2}\lim_{n\to\infty}\sum_{A\in 2^{\at(b_n)},\Vert A\Vert\geq 2}\prod_{a\in A}\epsilon^{2}\Vert \phi_a(g)\Vert^2=\epsilon^{-2} \lim_{n\to\infty}\prod_{a\in \at(b_n)}(1+\epsilon^2\Vert \phi_a(g)\Vert^2)-\epsilon^2\Vert g\Vert^2-1\\
&\leq \frac{e^{\epsilon^2\Vert g\Vert^2}-\epsilon^2\Vert g\Vert^2-1}{\epsilon^2}\to 0\text{ as }\epsilon\downarrow 0.
\end{align*}
Therefore, for any partition of unity $P$ of $\BBB$, for any $Q\subset P$, for all choices of $f\in \prod_{p\in Q}(\HHH(\{K=1\})\cap \HH_p)$, setting $X_p:=f_p$ for $p\in Q$ and $X_p:=\Omega$ for $p\in P\backslash Q$ we see that $\otimes_{p\in P}X_p$ belongs to the closure of the linear span of $\{\Exp(f):f\in \HHH(\{K=1\})\}$ by considering $\Exp(\oplus_{p\in Q}\epsilon_pf_p)=[\otimes_{p\in Q}\Exp(\epsilon_p f_p)]\otimes [\otimes_{p\in P\backslash Q}\Omega]$ for $\epsilon\in (0,\infty)^Q$, and taking a series of limits of linear combinations. Now the requisite totality follows from Proposition~\ref{proposition:products-total}.
\end{proof}
 \begin{corollary}\label{CE}
\ref{classical:C} implies  \ref{classical:E}.
 \end{corollary}
 \begin{proof}
Proposition~\ref{prop:Exp}, Items~\ref{prop:Exp-3} and~\ref{prop:Exp-5}.
 \end{proof}
Recall the notation of Definition~\ref{definition:noise-factorization} and that $\mu^{(1)}=\mu\vert_{\{K=1\}}$ is the restriction of $\mu$ to $\{K=1\}$.
\begin{proposition} \label{proposition:fock-spectrum}
 There exists a unique unitary isomorphism between  $\boldsymbol{\Upgamma}_s(\int^\oplus\GG_s\mu^{(1)}(\dd s))$ and $\HHH(\{K<\infty\})$, which maps $e_{\int^\oplus\GG_s\mu^{(1)}(\dd s)}(g)$ to $\Exp(\Psi^{-1}(g))$ for all $g\in \int^\oplus\GG_s\mu^{(1)}(\dd s)$.
\end{proposition}
\begin{proof}
It is well-known that the exponential vectors are total in $\Gamma_s(\int^\oplus\GG_s\mu^{(1)}_c(\dd s))$ and one verifies easily that this is true also of $\gamma_s(\oplus_{a\in \at(\mu^{(1)})}\mu^{(1)}(\{a\})\GG_a)$ (it is indeed implicit in the observation concerning the unitary equivalence of Examples~\ref{example:factorizations:discrete-Fock}  and~\ref{example:discrete-factorization} that was made on p.~\pageref{page:obseervation}). We conclude that the exponential vectors are total in $\boldsymbol{\Upgamma}_s(\int^\oplus\GG_s\mu^{(1)}(\dd s))$. Thus, by  Proposition~\ref{prop:Exp}, Items~\ref{prop:Exp-4}  and~\ref{prop:Exp-5}, and by \eqref{equation:scalar-product-in-exponential-Fock-space}, the stipulated association is scalar product-preserving between total sets, therefore it extends uniquely to a unitary isomorphism.
\end{proof}

 \begin{corollary}\label{CD}
\ref{classical:C} implies \ref{classical:D}.
 \end{corollary}
 \begin{proof}
 For all $x\in \BBB$, ``localizing'' everything to  $x$ in the manner of the first paragraph of Subsection~\ref{subsection:spectral-everything}, noting that the counting map associated to $(\BBB_x,\Omega)$ and $(\mu_x,\Phi_x)$ is given by $K\vert_{S_x}$ (a.e.-$\mu_x$),  and then applying Proposition~\ref{prop:Exp}\ref{prop:Exp-5}, gives that $\{\Exp(g):g\in \HHH(\{K=1\})\cap \HH_x\}$ is total in $\HH_x\cap \HHH(\{K<\infty\})=\HH_x$ (the equality thanks to \ref{classical:C}: $K<\infty$ a.e.-$\mu$). This, the multiplicative character of the exponential vectors of $\boldsymbol{\Upgamma}_s(\int^\oplus\GG_s\mu^{(1)}(\dd s))$ and Proposition~\ref{prop:Exp}\ref{prop:Exp-3} imply that the inverse of  the unitary isomorphism of  Proposition~\ref{proposition:fock-spectrum} carries $\Omega$ to $1$ and $\BBB$ onto a subfactorization of $\mathrm{Fock}(\int^\oplus\GG_s\mu^{(1)}(\dd s))$, indeed an $x\in \BBB$ onto $F_{\int^\oplus\GG_s\mu^{(1)}(\dd s)}(S_x\cap \{K=1\})$. But by definition this means that $(\BBB,\Omega)$ is isomorphic to a unital subfactorization of $(\mathrm{Fock}(\int^\oplus\GG_s\mu^{(1)}(\dd s)),1)$ and is thus of Fock type.
 \end{proof}
 Trivially
 \begin{proposition}\label{DA}
 \ref{classical:D} implies \ref{classical:A}.
 \end{proposition}
 \begin{proof}
 A Fock factorization  is complete. 
 \end{proof}
   \begin{remark}
 The proof of Corollary~\ref{CD} shows that the factors of a classical unital factorization $(\BBB,\Omega)$ can be indexed by the $\mu^{(1)}$-equivalence classes of the sets $S_x\cap \{K=1\}$, $x\in \BBB$; moreover, this indexation extends to an indexation by the measure algebra of $\mu^{(1)}$ of the factors of a complete unital factorization having $(\BBB,\Omega)$ as a subfactorization.
\end{remark}

\subsection{Multiplicative vectors are carried by the classical part}\label{subsection:multiplicative-carried}
We follow in this subsection quite closely \cite[Section~9]{vidmar-noise}, but trimmed down to our needs.

It is easy to connect multiplicative vectors to spectral independence probabilities:
\begin{proposition}\label{proposition:multiplicative-and-independence}
If $f$ is a multiplicative vector then $\mu_{f/\Vert f\Vert}$ is a spectral independence probability. 
\end{proposition}
\begin{proof}
For greater notational ease replace $f$ with $f/\Vert f\Vert$ (so now $\Vert f\Vert=1$  and $f/\langle \Omega,f\rangle$ is a multiplicative vector). Since $\langle \Omega,f\rangle\ne 0$, $\mu_{f}$ charges $\emptyset_S$.  Further, we have $\mu_f(S)=\Vert f\Vert^2=1$. Besides, for $x\in \BBB$, $u\in \BBB_x$ and $v\in \BBB_{x'}$, 
\begin{equation*}
\mu_f(S_{u\lor x'}\cap S_{x\lor v})=\mu_f(S_{u\lor v})=\Vert \phi_{u\lor v}(f)\Vert^2=\frac{\Vert\phi_u(f)\Vert^2\Vert\phi_v(f)\Vert^2}{\vert \langle\Omega,f\rangle\vert^2},
\end{equation*}
by the multiplicativity of $\frac{1}{\langle\Omega,f\rangle}\phi_{u\lor v}f$. Taking $u=x$, resp. $v=x'$, resp. $u=x$ and $v=x'$, we get $\mu_f(S_{x\lor v})=\frac{\Vert\phi_x(f)\Vert^2\Vert\phi_v(f)\Vert^2}{\vert \langle\Omega,f\rangle\vert^2}$, resp. $\mu_f(S_{u\lor x'})=\frac{\Vert\phi_u(f)\Vert^2\Vert\phi_{x'}(f)\Vert^2}{\vert \langle\Omega,f\rangle\vert^2}$, resp. $1=\frac{\Vert\phi_x(f)\Vert^2\Vert\phi_{x'}(f)\Vert^2}{\vert \langle\Omega,f\rangle\vert^2}$, which altogether allows to verify that $\mu_f(S_{u\lor x'}\cap S_{x\lor v})= \mu_f(S_{u\lor x'}) \mu_f(S_{x\lor v})$. Independence of $\pr_x$ and $\pr_{x'}$, i.e. of $\Sigma_x$ and $\Sigma_{x'}$ under $\mu_f$ now follows by an application of Dynkin's lemma (the general principle of which was noted in the comment immediately following the statement of Proposition~\ref{proposition:raise-independence}). 
\end{proof}

 \begin{proposition}\label{CF}
\ref{classical:C} implies \ref{classical:F}.
 \end{proposition}
 \begin{proof}
Thanks e.g. to the existence of orthonormal frames in direct integrals \cite[Proposition~II.1.1(ii)]{dixmier1981neumann} there is $f\in \HHH(\{K=1\})$ for which $\Psi(f)\ne 0$ a.e.-$\mu$ on $\{K=1\}$. The latter ensures that, in the notation of Subsection~\ref{subsection:spectrum-of-Fock}, $E(\Psi(f)\vert_{\{K=1\}})\ne 0$  a.e.-$\widetilde{\mu^{(1)}}$. From Propositions~\ref{prop:Exp}\ref{prop:Exp-3} and~\ref{proposition:multiplicative-and-independence}, and from $f\in \HHH(\{K=1\})$, we deduce that $\mu_{\Exp(f)/\Vert \Exp(f)\Vert}$ is a spectral independence probability. Finally, Proposition~\ref{proposition:fock-spectrum}, the proof of Corollary~\ref{CD}, the discussion of Subsection~\ref{subsection:spectrum-of-Fock} and the ``uniqueness of spectral resolutions of abelian von Neumann algebras'' \cite[A.85]{dixmier-c-star} yield that the property $E(\Psi(f)\vert_{\{K=1\}})\ne 0$  a.e.-$\widetilde{\mu^{(1)}}$ renders also  $\Psi(\Exp(f))\ne 0$ a.e.-$\mu$. This in turn results in $\mu_{\Exp(f)/\Vert \Exp(f)\Vert}$ being equivalent to $\mu$. 
 \end{proof}

The result of Proposition~\ref{theorem:new-conditioon-for-classicality} to follow is key. We quote beforehand an observation of purely probabilistic flavour,  which shall be applied in its proof. 

\begin{lemma}\label{proposition:random-set-domincance}
Let $n\in \mathbb{N}$. Under a probability $\PP$ let $\Gamma$ be a random subset of $[n]$ which excludes each $i\in [n]$ with probability $q_i$ independently of the others (in symbols, ${(\mathbbm{1}_{\{i\notin\Gamma\}})_{i\in [n]}}_\star\PP=\times_{i\in [n]}((1-q_i)\delta_0+q_i\delta_1)$); similarly let $\tilde{\Gamma}$ be a random subset of $[n]$ which excludes each $i\in [n]$ with probability $\sqrt[n]{q_1\cdots q_n}$ independently of the others. Then $\vert \Gamma\vert\leq \vert\tilde{\Gamma}\vert$ in first-order stochastic dominance. \cite[Proposition~9.9(ii)]{vidmar-noise}\qed
\end{lemma}

 \begin{proposition}\label{theorem:new-conditioon-for-classicality}
Let $\nu$ be a spectral independence probability. Then $\nu(K<\infty)=1$. 
\end{proposition}
\begin{proof}
The case of $\nu(\{\emptyset_S\})=1$, i.e. $\nu=\delta_{\emptyset_S}$ (in particular, the case when $0_\HH=1_\HH$) is trivial and excluded. 


Let $(b_n)_{n\in \mathbb{N}}$ be an exhausting sequence in $\mathfrak{F}$; in addition we put $b_0:=\{0_\HH,1_\HH\}$.


Set $T:=\cup_{n\in \mathbb{N}_0}\{n\}\times \at(b_n)$. The set $T$ is endowed  with a natural rooted tree structure: for $n\in \mathbb{N}_0$, $(n+1,b)$ is connected to $(n,a)$ iff $b\subset a$, no other connections; root $\{0\}\times \{1_\HH\}$. We define a $T$-indexed Bernoulli process $(\LL_t)_{t\in T}$ under the probability $\nu$ (in other words, a $\nu$-random subset of $T$): 
\begin{equation*}
\LL_{(n,a)}:=\mathbbm{1}_{\{\pr_a\ne \emptyset_S\}}\text{ for }(n,a)\in T.
\end{equation*}
 Then $\LL_{(0,1_\HH)}=\mathbbm{1}_{S\backslash \{\emptyset_S\}}$ a.e.-$\mu$; also, for each $(n,a)\in T$, $\mu$-a.e., 
 \begin{equation*}
 \{\LL_{(n,a)}=0\}=\{\pr_a=\emptyset_S\}=S_{a'}=\cap_{b\in \at(b_{n+1})\cap 2^a}S_{b'}=\cap_{b\in \at(b_{n+1})\cap 2^a}\{\pr_b=\emptyset_S\}=\cap_{b\in \at(b_{n+1})\cap 2^a}\{\LL_{(n+1,b)}=0\},
 \end{equation*} i.e.  $\mu$-a.e. a vertex of the tree is labeled zero by $\LL$ (is ``excluded''  from the random subset in question) iff all its descendants are. Besides, at each tree level $n\in \mathbb{N}_0$, the Bernoulli random variables $\mathbbm{1}_{(n,a)}$, $a\in\at(b_n)$, are independent under $\nu$ and the $\nu$-probability that all of them are zero at once is the constant $p_0:=\nu(\{\emptyset_S\})\in(0,1)$. It is helpful to think of $\LL$ as an exploration process: as $n$ increases we pierce with more and more precision into the structure of $\BBB$, at least as far as the counting map $K$ is concerned.

Now, for $n\in \mathbb{N}_0$, setting $\Gamma_n:=\{(n,a):a\in \at(b_n),\LL_{(n,a)}=1\}$, which is the random set of points of $T$ at level $n$ that are included by the exploration process $\LL$, we have
\begin{equation*}
\vert\Gamma_n\vert=\sum_{a\in \at(b_n)}\mathbbm{1}_{\{\pr_a\ne \emptyset_S\}}=\sum_{a\in \at(b_n)}\mathbbm{1}_{S\backslash S_{a'}}=K_{b_n}
\end{equation*}
a.e.-$\mu$. Therefore, to show that $\nu(K<\infty)=1$, by continuity of probability, it suffices (and is indeed equivalent) to prove that 
\begin{equation*}\lim_{m\to\infty}\lim_{n\to\infty}\nu(\vert\Gamma_n\vert\leq m)=1.
\end{equation*}
By Lemma~\ref{proposition:random-set-domincance} it reduces to establishing that 
$\lim_{m\to\infty}\lim_{k\to\infty}\PP_k(\vert\Gamma\vert\leq m)=1$, where, for $k\in \mathbb{N}$, under $\PP_k$, $\Gamma$ is a random subset of $[k]$ that excludes each point with probability $\sqrt[k]{p_0}$ independently of the others. But for $\{k,m\}\subset \mathbb{N}$, 
\begin{align*}
\PP_k(\vert\Gamma\vert\leq m)&=\sum_{l=0}^m{k\choose l}(1-\sqrt[k]{p_0})^l\sqrt[k]{p_0}^{k-l}=p_0\sum_{l=0}^m{k\choose l}(\sqrt[k]{p_0}^{-1}-1)^l\\
&=p_0\sum_{l=0}^m\frac{1}{l!}k(p_0^{-1/k}-1)\cdots (k-l+1)(p_0^{-1/k}-1)\xrightarrow[]{k\to\infty} p_0\sum_{l=0}^m\frac {(-\log p_0)^l}{l!}\xrightarrow[]{m\to\infty} p_0e^{-\log p_0}=1,
\end{align*}
(the limit ``$k\to\infty$''' uses e.g. l'H\^ospital's rule).
\end{proof}

\begin{corollary}\label{EF,C}
[\ref{classical:E} or \ref{classical:F}] implies \ref{classical:C}.
\end{corollary}
\begin{proof}
Combine Propositions~\ref{proposition:multiplicative-and-independence} and~\ref{theorem:new-conditioon-for-classicality}.
\end{proof}

\subsection{A digression: additive vectors and blackness; the classical part}\label{subsection:blackness}
The present subsection is somewhat out of the main line of sight of our considerations. Still, the results to feature are at our fingertips and seem worth recording given their importance in the commutative setting.
 
 In analogy to the definition of a multiplicative vector (Definition~\ref{definition:multiplicative-vector}) we have

\begin{definition}
An additive vector is an element $\xi$ of $\HH$  such that for all $x\in \BBB$, $\xi=\phi_x(\xi)+\phi_{x'}(\xi)$.
\end{definition}
$0$ is always an additive vector. For any additive vector $\xi$ it is immediate from the definition  that $\langle \Omega,\xi\rangle=0$, also that $\phi_x(\xi)$ is an additive vector in turn for all $x\in \BBB$, finally that $\xi=\oplus_{p\in P}\phi_p(\xi)$ for all partitions of unity $P$ of $\BBB$.
\begin{example}\label{example:noise-Boolean-additive}
In the context of Example~\ref{example:noise-boolean} an additive vector of $(B^\uparrow,\mathbbm{1})$ is a square-integrable additive integral of $B$, i.e. a $\xi\in \LLL^2(\PP)$ such that $\xi=\PP[\xi\vert x]+\PP[\xi\vert x']$ for all $x\in B$. 
\end{example}

\begin{proposition}\label{proposition:first-chaos}
$\HHH(\{K=1\})$ comprises precisely the additive vectors.
\end{proposition}
\begin{proof}
Let $\xi$ be an additive vector. Then for all $b\in \mathfrak{F}$, $\Psi(\xi)=(\sum_{x\in \at(b)}\mathbbm{1}_{S_x})\Psi(\xi)$ a.e.-$\mu$, i.e. $1=\sum_{x\in \at(b)}\mathbbm{1}_{S_x}$ a.e.-$\mu$ on $\{\Psi(\xi)\ne 0\}$ for all $x\in \BBB$. Taking an exhausting sequence $(b_n)_{n\in \mathbb{N}}$ in $\mathfrak{F}$ we have $1=\sum_{x\in \at(b_n)}\mathbbm{1}_{S_x}$ simultaneously for all $n\in \mathbb{N}$ a.e.-$\mu$ on  $\{\Psi(\xi)\ne 0\}$, which means that $K=1$ a.e.-$\mu$ on  $\{\Psi(\xi)\ne 0\}$, yielding $\xi\in \HHH(\{K=1\})$. The converse appeared already in Proposition~\ref{prop:Exp}\ref{prop:Exp-1}. 
\end{proof}

The theorem to feature next concerns the situation that is as singular to the classical one described in Theorem~\ref{thm:main-for-noise-factorization} as can be.  It is only interesting when $\HH$ is not one-dimensional, i.e. $0_\HH\ne 1_\HH$ (otherwise all its statements hold true trivially). 
 \begin{theorem}\label{thm:black}
The following are equivalent. 
 \begin{enumerate}[(I)]
 \item\label{black:A} $\Omega$ is the only multiplicative vector.
 \item\label{black:B} $0$ is the only additive vector.
 \item\label{black:C} $\delta_{\emptyset_S}$ is the only spectral independence probability.
 \item\label{black:D} $K=\infty$ a.e.-$\mu$ off $\{\emptyset_S\}$. 
 \end{enumerate}
 \end{theorem}
 \begin{proof}
\ref{black:A} $\Rightarrow$ \ref{black:D}: Proposition~\ref{prop:Exp}, Items~\ref{prop:Exp-3} and~\ref{prop:Exp-5}. \ref{black:D} $\Rightarrow$ \ref{black:C}: Proposition~\ref{theorem:new-conditioon-for-classicality}. \ref{black:C} $\Rightarrow$ \ref{black:A}:   Proposition~\ref{proposition:multiplicative-and-independence}.  \ref{black:D} $\Rightarrow$ \ref{black:B}:   Proposition~\ref{proposition:first-chaos}. \ref{black:B} $\Rightarrow$ \ref{black:D}: Propositions~\ref{proposition:products-total} and~\ref{proposition:first-chaos}. 
 \end{proof}
 \begin{definition}
 $(\BBB,\Omega)$ is black if it meets one and then all of the conditions of Theorem~\ref{thm:black} but $0_\HH\ne 1_\HH$.
 \end{definition}
 Compare with Definition~\ref{definition:classical}.  Both are in line with the usage of these concepts for noise Boolean algebras \cite[Definition~1.3]{tsirelson} (due to \cite[Theorem~7.7]{tsirelson} in the classical case). We dwell for a while on this ``$\sigma$-field'' setting in

\begin{extendedremark}\label{extendedremark:noise-Boolean}
Let $\PP$ be an essentially separable probability, i.e. such that $\LLL^2(\PP)$ is separable. Recalling the content of Examples~\ref{example:noise-boolean},~\ref{example:noise-Boolean-bis},~\ref{example:noise-boolean-contd} and~\ref{example:noise-Boolean-additive}, let $\HH=\LLL^2(\PP)$ and $(\BBB,\Omega)=(B^\uparrow,\mathbbm{1})$. 

The following are equivalent (to the classicality of $B$ \cite[Definition~1.3(a)]{tsirelson}). 
\begin{enumerate}[(A')]
\item\label{classical:a} There is a complete (:= closed for arbitrary joins and meets) noise Boolean algebra of $\PP$ containing $B$. 
\item\label{classical:b}  For every sequence $(x_n)_{n\in \mathbb{N}}$ in $B$, $(\lor_{n\in \mathbb{N}}x_n)\lor (\land_{n\in \mathbb{N}}x_n')=1_\PP$.
\item\label{classical:c} $K<\infty$ a.e.-$\mu$.
\item\label{classical:d} $(B^\uparrow,\mathbbm{1})$ is of Fock type.
\item\label{classical:e} The square-integrable  multiplicative integrals of $B$, together with $0_\PP$, generate $1_\PP$.
\item\label{classical:f} There exists a spectral independence probability equivalent to $\mu$.
\item\label{classical:g} The square-integrable additive integrals of $B$,  together with $0_\PP$, generate $1_\PP$.
\end{enumerate}
\begin{proof}
The equivalence of \ref{classical:a}-\ref{classical:b}-\ref{classical:c}-\ref{classical:g} was established in \cite[Theorems~1.5 and~7.7]{tsirelson} (modulo a small detail overlooked in the proof of \cite[Corollary~4.7]{tsirelson}, which is however easily filled in \cite[Remark~4.10]{vidmar-noise}), while the equivalence of \ref{classical:c}-\ref{classical:d}-\ref{classical:f} is just a specialization of Theorem~\ref{thm:main-for-noise-factorization}. It remains to check e.g. that \ref{classical:e} is equiveridical with \ref{classical:g}, which will certainly follow if  it can be argued that 
\begin{align*}
\sigma_{\mathrm{add}}(B)&:=\sigma(\{\xi\in \LLL^2(\PP):\xi\text{ an additive integral of }B\})\lor 0_\PP\\
&=\sigma(\{\xi\in \LLL^2(\PP):\xi\text{ a multiplicative integral of }B\})\lor 0_\PP=:\sigma_{\mathrm{mult}}(B).
\end{align*}
But if $\xi$ is a square-integrable additive integral of $B$, which thanks to Proposition~\ref{proposition:first-chaos} is equivalent to $\xi\in \HHH(\{K=1\})$, then $\xi=\lim_{\epsilon\downarrow 0}\frac{\Exp(\epsilon \xi)-1}{\epsilon}$ (see proof of  Proposition~\ref{prop:Exp}\ref{prop:Exp-5}); combined with Proposition~\ref{prop:Exp}\ref{prop:Exp-3} we get $\sigma_{\mathrm{add}}(B)\subset\sigma_{\mathrm{mult}}(B)$. For the converse it suffices to verify that ($\star$) there is a (automatically then unique) complete sub-$\sigma$-field $\GG$ of $1_\PP$ for which $\HHH(\{K<\infty\})=\LLL^2(\PP\vert_\GG)$. For  Propositions~\ref{proposition:multiplicative-and-independence} and~\ref{theorem:new-conditioon-for-classicality} will then render $\sigma_{\mathrm{mult}}(B)\subset \GG$, while Proposition~\ref{proposition:products-total} ensures that $\GG\subset \sigma_{\mathrm{add}}(B)$; altogether it will deliver
\begin{equation*}
\sigma_{\mathrm{stb}}(B):=\sigma_{\mathrm{add}}(B)=\sigma_{\mathrm{mult}}(B)\text{ and }\HHH(\{K<\infty\})=\LLL^2(\PP\vert_{\sigma_{\mathrm{stb}}(B)}).
\end{equation*}
Now, in order to establish $(\star)$ we need \cite[Fact~2.1]{tsirelson} only show that $\HHH(\{K<\infty\})$ is closed under applications of the absolute value. Take then $f\in \HHH(\{K<\infty\})$. With $(b_n)_{n\in \mathbb{N}}$ an exhausting sequence, $\mathbbm{1}_{\{K<\infty\}}=\lim_{\rho\uparrow 1}\lim_{n\to\infty}\rho^{K_{b_n}}=\lim_{\rho\uparrow 1}\lim_{n\to\infty}\prod_{a\in \at(b_n)}((1-\rho)\mathbbm{1}_{S_{a'}}+\rho)$, which shows that $\mathbbm{1}_{\{K<\infty\}}$ belongs to the smallest convex subset $\mathcal{S}$ of $\LLL^0(\mu)$ containing all the $\mathbbm{1}_{S_x}$, $x\in B$, and closed under pointwise a.e.-$\mu$ limits, such $\mathcal{S}$ being then automatically closed under multiplication thanks to \eqref{spectral-sets:pi-system}. But for $x\in B$, $\mu_{\vert f\vert}(S\backslash S_x)=\PP[f^2]-\PP[\PP[f\vert x]^2]=\PP[\mathrm{var}(\vert f\vert \vert x)]\leq \PP[\mathrm{var}( f\vert x)]=\mu_f(S\backslash S_x)$ ($\because$ the contraction $\vert\cdot\vert$ reduces  [conditional] variance); by linearity and bounded convergence we see that $\mu_{\vert f\vert}[1-\phi]\leq \mu_f[1-\phi]$ for all $\phi\in \mathcal{S}$, and  in particular with $\phi=\mathbbm{1}_{\{K<\infty\}}$ we  get $\mu_{\vert f\vert}(K=\infty)\leq \mu_f(K=\infty)=0$, i.e. $\vert f\vert \in \HHH(\{K<\infty\})$.
\end{proof}
We see that \ref{classical:a}-\ref{classical:f} are exactly parallel to \ref{classical:A}-\ref{classical:F} of Theorem~\ref{thm:main-for-noise-factorization}, while \ref{classical:g} is more specific to $\sigma$-fields.  
\end{extendedremark}
Returning now to the general setting,  \eqref{projections-and-subspaces} with $E_p=\{K<\infty\}$ for all $p\in P$ and \eqref{K-local} yield
\begin{equation*}
\HH^\star:=\HHH(\{K<\infty\})=\otimes_{p\in P}(\HH^\star\cap \HH_p)
\end{equation*}
for all partitions of unity $P$ of $\BBB$; in particular, for all $x\in \BBB$, $\HH^\star=(\HH^\star\cap \HH_x)\otimes (\HH^\star\cap \HH_{x'})$. This means that the map 
\begin{equation*}
\BBB\ni x\mapsto x^\star:=1_{\HH^\star\cap \HH_x}\otimes 0_{\HH^\star\cap \HH_{x'}}\in \widehat{\HH^\star}
\end{equation*} carries $\BBB$ onto a type $\I$ factorization $\BBB^\star$ of $\HH^\star$ as a homomorphism of Boolean algebras. Moreover, $(\BBB^\star,\Omega)$ is a unital factorization, and, because $\{K<\infty\}\in \Sigma$, the pair $(\mu^\star,\Psi^\star):=(\mu\vert_{\{K<\infty\}},\Psi\vert_{\HH^\star}\vert_{\{K<\infty\}})$ is a spectral resolution thereof with spectral sets $S_{x^\star}^\star=S_x\cap \{K<\infty\}$, $x\in \BBB$, and (hence) counting map $K^\star=K\vert_{\{K<\infty\}}<\infty$ a.e.-$\mu^\star$. It follows that $(\BBB^\star,\Omega)$ is a classical unital factorization, equal to  $(\BBB,\Omega)$ or $(\{0_\HH,1_\HH\},\Omega)$ according as to whether $(\BBB,\Omega)$ is classical or black. In the setting of Extended remark~\ref{extendedremark:noise-Boolean} it corresponds to 
\begin{equation*}
B\ni x\mapsto x\land \sigma_{\mathrm{stb}}(B)\in \reallywidehat{\PP\vert_{\sigma_{\mathrm{stb}}(B)}}
\end{equation*} mapping $B$ onto a classical noise Boolean algebra $B_{\mathrm{stb}}$ of $\PP\vert_{\sigma_{\mathrm{stb}}(B)}$ (again as a homomorphism of Boolean algebras) --- $B_{\mathrm{stb}}$ is the so-called classical or stable part of $B$. For more on the latter see \cite[Section~7]{vidmar-noise} and the references therein.
%

\section{Complete type $\mathrm{I}$ factorizations}\label{section:factorizations-sans}
In this section 
\begin{equation*}
\text{$\KK$ is a  non-zero Hilbert space and $\FFF$ is a type $\mathrm{I}$ factorization of $\KK$.}
\end{equation*}

 Proposition~\ref{thm:a-facorizable-vector} to feature presently establishes that any complete type $\I$ factorization acting on a non-zero separable Hilbert space  can in fact be enhanced into a unital factorization. Before proceeding to this let us recall some notions and facts from \cite{araki-woods} that we shall avail ourselves of in the proof of Proposition~\ref{thm:a-facorizable-vector}. 
 
First  is the deep structural result of \cite[Theorem~4.1]{araki-woods} stating that a complete atomic type $\I$ factorization $\FFF$ is isomorphic to a factorization of the form $\prod(\otimes_{\beta\in \mathfrak{B}}\JJ_\beta^{e_\beta})$ of Example~\ref{example:discrete-factorization} with $\mathfrak{B}$ the collection of the atoms of $\FFF$, and it is so via  a unitary isomorphism that sends an atom $a$ of $\FFF$ onto $F_{\{a\}}:=$ (the ``copy'' of $1_{\JJ_a}$ in $\otimes_{\beta\in \mathfrak{B}}\JJ_\beta$) $:=1_{\JJ_a}\otimes 0_{\JJ_{\mathfrak{B}\backslash \{a\}}}$ (up to the natural unitary equivalence), the minimal non-zero projections of $a$ corresponding to copies of the one-dimensional projections of $\JJ_a$. In particular,  a finite type $\I$ factorization is always isomorphic to one of those given by Example~\ref{example:discrete-factorization} with $\mathfrak{B}$ finite \cite[Section~2]{araki-woods}, the presence of the stabilising family being then superfluous.   Compare this with $(\otimes)$ of the second paragraph of Subsection~\ref{subsection:factorizable}: it corresponds to $\FFF=\{0_\KK,x,x',1_\KK\}$ for a type $\I$ factor $x$ of $\KK$.

For  $\psi\in \KK$ we set next \cite[Definition~6.1 and Eq.~(6.1)]{araki-woods}
\begin{equation*}
\dd(\psi;\FFF):=\inf_P\sup_Q\left\langle \psi,\prod_{f\in P}Q_f\psi\right\rangle,
\end{equation*}
where $P$ runs over all  partitions of unity of (the Boolean algebra) $\FFF$ and then $Q$ over all choice functions on $P$ having as $Q(p)$ a minimal non-zero projection of $p$ for all $p\in P$. Notice that $\dd(\psi;\FFF)\leq \Vert\psi\Vert^2$ for all $\psi\in\KK$; and that if $\psi$ is a norm one factorizable vector of $\FFF$, then thanks to Corollary~\ref{corollary:factorizable-and-partition-of-unity} $\dd(\psi;\FFF)=1$. 

A partition of $\FFF$ is defined naturally: it is a subset $P$ of $\FFF\backslash \{0_\KK\}$ such that $\lor P=1_\KK$ and such that $x\land y=0_\KK$ for $x\ne y$ from $\FFF$. Thus a finite partition of $\FFF$ is nothing but a partition of unity of (the Boolean algebra) $\FFF$. However,  partitions of $\FFF$ can be infinite. 

Assume now until the end of this paragraph that $\FFF$ is complete nonatomic. Let $P$ be a  partition of $\FFF$. For a family of Hilbert spaces $(\KK_p)_{p\in P}$ of dimension at least two and  relative to a stabilising family $\zeta=(\zeta_p)_{p\in P}$ of norm one vectors, up to a unitary isomorphism $\Theta_P$, we have that $\KK=\otimes_{p\in P}\KK_p$ turning each $p\in P$ into the copy of $1_{\KK_p}$ in $\otimes_{p\in P}\KK_p$ and transforming $\FFF\cap 2^p$ into a complete nonatomic type $\I$ factorization  of $\KK_p$ in the natural way, namely into
\begin{equation}\label{naturally-transformed}
\FFF_{p}:=\{\{X\in 1_{\KK_p}:X\otimes \mathbf{1}_{\otimes_{q\in P\backslash \{p\}}\KK_q}\in x\}:x\in \FFF\cap 2^{p}\}.
\end{equation}
By \cite[Lemma~6.1]{araki-woods} applied with $\Psi_p=\zeta_p$ for $p\in P$  (here $\Psi_p$ is the notation of \cite[Lemma~6.1]{araki-woods}),
\begin{equation}\label{adiuvat:one}
\text{if $P$ is countable then }\sum_{p\in P}(1-\dd(\zeta_p;\FFF_p))<\infty,\text{ in particular }\dd(\zeta_p;\FFF_p)>0 \text{ for all but finitely many }p\in P.
\end{equation}
(The assumption $\otimes_{p\in P} \Psi_p\ne 0$ is missing (but  needed) in \cite[Lemma~6.1]{araki-woods}, however it is a trivial omission/it is implicitly understood from context.) 

Another observation of \cite{araki-woods} we formulate, again under the assumption that $\FFF$ is complete nonatomic, as follows: for a norm one $\psi\in \KK$,
\begin{equation}\label{adiuvat:two}
\text{if $\dd(\psi;\FFF)>0$ then $\forall\delta\in (0,1)$  $\exists$ a factorizable vector $\phi$ of $\FFF$ satisfying $\langle \psi,\phi\rangle \geq \delta \dd(\psi;\FFF)$ and $\Vert \phi\Vert\leq 1$}.
\end{equation}
This follows from the \emph{proof} of \cite[Lemma~6.3]{araki-woods} after trivial modifications (namely, in \cite[Eq.~(6.3)]{araki-woods} one can ask for, in the notation thereof, ``$>\delta\dd(\Psi;B)$'' rather than ``$>\frac{1}{2}\dd(\Psi;B)$''). (A consequence of this is the \emph{statement} of \cite[Lemma~6.3]{araki-woods} asserting, in our notation, that if  $\dd(\psi;\FFF)>0$ for all $\psi\in \KK\backslash \{0\}$, then the factorizable vectors are total in $\KK$. Incidentally, the ``$\backslash\{0\}$'' is missing (but should be present) in \cite[Lemma~6.3]{araki-woods}, however again this is a trivial omission/is implicit from context.)

We are  ready to establish
\begin{proposition}\label{thm:a-facorizable-vector}
Suppose $\KK$ is separable. If $\FFF$ is complete  then it admits a  factorizable vector.
\end{proposition}
Completeness is essential: in \cite[Example on p.~91]{vershik-tsirelson} is described  a type $\I$ factorization on a non-zero separable Hilbert space having no  factorizable vectors, which therefore  cannot be upgraded to a unital factorization and cannot be complete. On the other hand, whether or not one can do away with the assumption of separability of the underlying Hilbert space, while still retaining the validity of Proposition~\ref{thm:a-facorizable-vector}, and more generally of the results of this paper (perhaps properly reinterpreted), is a question that lies  beyond the scope of this investigation (the author is not aware of any counterexample, nor would he feel confident making a conjecture either way).
\begin{proof}
Consider $x_0:=\lor\at(\FFF)\in \FFF$, the join of all the atoms of $\FFF$, which belongs to $\FFF$ by completeness. Up to a unitary isomorphism $\Delta_{x_0}$ we have a decomposition $\KK=\KK_0\otimes \KK_0'$ into the tensor product of two non-zero Hilbert spaces such that $x_0=1_{\KK_0}\otimes 0_{\KK_0'}$.

The unitary isomorphism $\Delta_{x_0}$ transforms $\FFF\cap 2^{x_0}$ into a complete atomic type $\I$ factorization $\FFF_{x_0}$ of $\KK_0$ in the natural way, $\FFF_{x_0}=\{\{X_0\in 1_{\KK_0}:X_0\otimes \mathbf{1}_{\KK_0'}\in x\}:x\in \FFF\cap 2^{x_0}\}$.  We know $\FFF_{x_0}$ is isomorphic to a discrete Fock factorization, therefore certainly admits a factorizable vector.

On the other hand, $\FFF_{x_0'}$ corresponding to $\FFF\cap 2^{x_0'}$ is a complete nonatomic type $\I$  factorization of $\KK_0'$. To ease the notation we may and do assume at this point that $\FFF$ was atomless to begin with, forget about $x_0$, and seek to show that it admits a  factorizable vector. 

Some more vocabulary. A $Q\subset \FFF\backslash \{0_\KK\}$ having $x\land y=0_\KK$ for all $x\ne y$ from $Q$ we call  a partial partition of $\FFF$ (note: no restriction on the cardinality of $Q$). 
We say such $Q$ is factorizing if 
for each $q\in Q$ there exists a minimal non-zero projection $D\in q$ which decomposes into the product $D=AA'$ for some $A\in q\land x$ and some $A'\in q\land x'$ for all $x\in \FFF$. 

Consider now the collection 
\begin{equation*}
\mathfrak{P}:=\{Q\in 2^{\FFF\backslash \{0_\KK\}}:Q\text{ is a factorizing partial partition of $\FFF$}\}.
\end{equation*}
We partially order $\mathfrak{P}$ by inclusion, note that $\mathfrak{P}$ is non-empty (since $\emptyset \in \mathfrak{P}$) and that every linearly ordered subset of $\mathfrak{P}$ admits an upper bound, namely its union. By Zorn's lemma there is $P\in \mathfrak{P}$ maximal w.r.t. inclusion. 

We claim that $\lor P=1_\KK$. Suppose to the contrary. By completeness $\lor P\in \FFF$, hence there are non-zero Hilbert spaces $\KK_P$ and $\KK_{P'}$ for which, up to a unitary isomorphism $\Gamma_P$, $\KK=\KK_P\otimes \KK_{P'}$ and $\lor P=1_{\KK_P}\otimes 0_{\KK_{P'}}$. Since $\lor P\ne 1_\KK$, $\KK_{P'}$ is at least two-dimensional. The type $\I$ factorization $\FFF_{P'}$   of $\KK_{P'}$ got from $\FFF\cap 2^{(\lor P)'}$ via $\Gamma_P$ is atomless and complete. Since $1_{\KK_{P'}}\ne 0_{\KK_{P'}}$ we deduce existence of a countably infinite partition  $Z'$ of $\FFF_{P'}$, corresponding via $\Gamma_P^{-1}$ to a partial partition $Z$ of $\FFF$ having $\lor Z\subset (\lor P)'$, a fortiori $Z\cap P= \emptyset$. By \eqref{adiuvat:one} applied to the partition $Z'$ of $\FFF_{P'}$  and then by \eqref{adiuvat:two} with $\delta=\frac{1}{2}$ (say)
there is a non-empty (even infinite and moreover cofinite) subset $R$ of $Z$ such that $P\cup R$ is a factorizing partial partition of $\FFF$. This contradicts the maximality of $P$. We have thus established that $P$ is a partition of $\FFF$. 

We are now in the setting and notation of the  paragraph surrounding \eqref{naturally-transformed}, in particular we have access to $\Phi_P$ and the associated identifications. 

By the factorization property of $P$, for each $p\in P$, $\FFF_p$ admits a factorizable vector $\chi_p\in \KK_p$ of norm one, i.e. the copy $D_p$ of $\vert \chi_p\rangle\langle\chi_p\vert$ in $\otimes_{p\in P}\KK_p$ is, up to $\Theta_P$, a non-zero minimal projection of $p$  which decomposes into the product $D_p=AA'$ for some $A\in p\land x$ and some $A'\in p\land x'$ for all $x\in \FFF$. We improve on the family $(\chi_p)_{p\in P}$ as follows.  Since $\KK$ is separable, $P$ is countable. Referring again to \eqref{adiuvat:one} we get that $\sum_{p\in P}(1-\dd(\zeta_p;\FFF_p))<\infty$, in particular that the set \begin{equation*}
P':=\{p\in P:\dd(\zeta_p;\FFF_p)>0\}
\end{equation*}
 is cofinite in $P$. Seeing as how $P$, a fortiori $P'$ is countable, there exists a $\delta:P'\to (0,1)$ satisfying 
\begin{equation*}
\sum_{p\in P'}(1-\delta_p\dd(\zeta_p;\FFF_p))<\infty.
\end{equation*}
 Applying \eqref{adiuvat:two} we get a family $(\phi_p)_{p\in P'}$ having the property that for each $p\in P'$, $\phi_p$ is a factorizable vector of $\FFF_p$ verifying $\langle \zeta_p,\phi_p\rangle\geq \delta_p\dd(\zeta_p;\FFF_p)$ and $\Vert \phi_p\Vert\leq 1$. With this family in hand, if we set 
\begin{equation*}
\chi_p':=
\begin{cases}
\chi_p &\text{for } p\in P\backslash P'\\
\phi_p  & \text{for }p\in P'
\end{cases},
\end{equation*}
then 
\begin{enumerate}[(i)]
\item\label{property:1} for all $p\in P$, $\chi'_p$  is a factorizable vector of $\FFF_p$, 
\item\label{property:2} $\sum_{p\in P}\vert 1-\langle \zeta_p,\chi_p'\rangle\vert<\infty$ (due to $\langle \zeta_p,\chi_p'\rangle\geq \delta_p\dd(\zeta_p;\FFF_p)$ for $p\in P'$ and in virtue of $\sum_{p\in P'}(1-\delta_p\dd(\zeta_p;\FFF_p))<\infty$) and (hence) 
\item\label{property:3} $0< \prod_{p\in P}\vert\langle \zeta_p,\chi_p'\rangle\vert \leq \prod_{p\in P}\Vert \chi'_p\Vert\leq 1<\infty$ (since $\Vert \phi_p\Vert\leq 1$ for $p\in P'$,  $\Vert \chi_p \Vert=1$ for $p\in P\backslash P'$, and by Cauchy-Schwartz).
\end{enumerate}

We claim next that
\begin{equation*}
\omega:=\otimes_{p\in P} \chi'_p:=\lim_{F\in (2^P)_{\mathrm{fin}}}\otimes (\chi'\vert_F\cup \zeta\vert_{P\backslash F})
\end{equation*}
 is a well-defined non-zero vector of $\otimes_{p\in P}\KK_p$, the limit being w.r.t. the partial order of inclusion on $(2^P)_{\mathrm{fin}}$. 
To see why, take $F\subset G$ finite subsets of $P$, and evaluate
\begin{equation*}
\Vert \otimes(\chi'\vert_F\cup \zeta\vert_{P\backslash F})-\otimes(\chi'\vert_G\cup \zeta\vert_{P\backslash G})\Vert^2=\left(\prod_{f\in F}\Vert\chi'_f\Vert^2\right)\left(\left(-1+\prod_{f\in G\backslash F}\Vert \chi'_f\Vert^2\right)+2\Re\left(1-\prod_{f\in G\backslash F}\langle \zeta_f,\chi'_f\rangle\right)\right).
\end{equation*}
This expression is $\to 0$ as $F\uparrow P$ due to \ref{property:2}-\ref{property:3}. Thus the existence of the limit defining $\omega$ follows by completeness of $\KK$. Certainly $\omega$ is not zero, indeed $\Vert \omega \Vert=\prod_{p\in P}\Vert \chi'_p\Vert>0$. 

We may and do view $\omega$ also as a member of $\KK$, on passing through $\Theta_P$. We shall show that it is a factorizable vector. To this end take finally any $x\in \FFF$. Thanks to \ref{property:1} just above, for each $p\in P$, the copy $D'_p$ of $\vert \chi'_p\rangle\langle \chi'_p\vert$  in $\otimes_{p\in P}\KK_p$ is, up to $\Theta_P$, of the form $A_pA_p'$ for some $A_p\in p\land x$ and some $A_p'\in p\land x'$, which moreover by Corollary~\ref{corollary:factorizable-and-partition-of-unity} we may and do assume are projections. Then, still up to $\Theta_P$ and natural identifications,
\begin{equation*}
\vert \omega\rangle\langle \omega\vert=\stronglim_{F\in (2^P)_{\mathrm{fin}}}(\otimes_{p\in F}\vert \chi'_p\rangle\langle \chi'_p\vert)\otimes \mathbf{1}_{\otimes_{p\in P\backslash F}\KK_p}=\stronglim_{F\in (2^P)_{\mathrm{fin}}}\prod_{p\in F}A_pA_p'=\left(\stronglim_{F\in (2^P)_{\mathrm{fin}}}\prod_{p\in F}A_p\right)\left(\stronglim_{F\in (2^P)_{\mathrm{fin}}}\prod_{p\in F}A_p'\right),
\end{equation*}
s-lim being strong limits, the final two existing because they are even of noinincreasing nets of projections. But   $\stronglim_{F\in (2^P)_{\mathrm{fin}}}\prod_{p\in F}A_p\in x$ and $\stronglim_{F\in (2^P)_{\mathrm{fin}}}\prod_{p\in F}A_p'\in x'$, establishing that $\omega$ is a factorizable vector.
\end{proof}


\begin{theorem}\label{thm:main-two}
Suppose $\KK$ is separable, $\FFF$ complete. Then the factorizable vectors of $\FFF$ are total. 
\end{theorem}
 \begin{proof}
 Combine Proposition~\ref{thm:a-facorizable-vector},  Theorem~\ref{thm:main-for-noise-factorization}, \ref{classical:A} $\Rightarrow$ \ref{classical:E}, and Proposition~\ref{proposition:multiplicative-factorizable}, in this order. 
 \end{proof}
 \begin{remark}
 Actually, the same proof shows that, under the assumption of Theorem~\ref{thm:main-two}, for any given  factorizable vector $\omega$ of $\FFF$ (which exists), the factorizable vectors $\xi$ of $\FFF$ for which $\langle\omega,\xi\rangle=1$ are already total.
\end{remark}
\begin{corollary}\label{cor:main}
Suppose $\KK$ is separable. 
The following are equivalent. 
\begin{enumerate}[(A)]
\item \label{characterization:factorizations:A} $\FFF$ is complete.
\item \label{characterization:factorizations:C}$\FFF$ is isomorphic to a Fock factorization. 
\end{enumerate}
\end{corollary}
\begin{remark}
``The'' Fock factorization of Corollary~\ref{cor:main} is continuous or discrete respectively as $\FFF$ is atomless or atomic (it is both in the trivial case when $\KK$ is one-dimensional). The assumption that $\FFF$ is type $\I$ does not follow from completeness: in \cite[p.~183]{araki-woods} is described a complete atomic factorization of a non-zero separable Hilbert space not all of whose element are type $\I$ factors (some are type $\mathrm{II}$).
\end{remark}
\begin{proof}
A Fock factorization is complete so \ref{characterization:factorizations:C} implies \ref{characterization:factorizations:A}. The converse follows from Theorem~\ref{thm:main-two} and \cite[Theorems~4.1 and~6.1]{araki-woods}. 
\end{proof}
\bibliographystyle{plain}
\bibliography{QP}
\end{document}